%% file: MAIN_springer.tex
\RequirePackage{etoolbox,tikz,pgfplots}
\usepgfplotslibrary{fillbetween}
\patchcmd{\bibliographystyle}{#1}{spmpsci}{}{}

\documentclass[sn-mathphys]{sn-jnl}
\jyear{2025}

\input{preamble.tex}

\begin{document}

    \title[On the well-posedness of a nonlocal kinetic model for dilute polymers]{\vspace{-2.1cm} On the well-posedness of a nonlocal kinetic model for dilute polymers with anomalous diffusion}
    \author{\fnm{Marvin} \sur{Fritz$^1$}}
    \affil{\small\orgdiv{${}^1$Radon Institute for Computational and Applied Mathematics},  \orgaddress{\city{Linz}, \country{Austria}}}

    \author{\fnm{Endre} \sur{S\"{u}li$^2$}}
    \affil{\small\orgdiv{${}^2$Mathematical Institute}, \orgname{University of Oxford}, \orgaddress{\city{Oxford}, \country{United Kingdom}}}

    \author{\fnm{Barbara} \sur{Wohlmuth$^3$}}
    \affil{\small\orgdiv{${}^3$School of Computation, Information and Technology}, \orgname{Technical University of Munich}, \orgaddress{\city{Munich}, \country{Germany\vspace{-1.2cm}}}}
	
	\abstract{
		In this work, we study a class of nonlocal-in-time kinetic models of incompressible dilute polymeric fluids. The system couples a macroscopic balance of linear momentum equation with a mezoscopic subdiffusive Fokker--Planck equation governing the evolution of the probability density function of polymer configurations. The model incorporates nonlocal features to capture subdiffusive and memory-type phenomena. Our main result asserts the existence of global-in-time large-data weak solutions to this nonlocal system. The proof relies on an energy estimate involving a suitable relative entropy, which enables us to handle the critical general non-corotational drag term that couples the two equations. As a side result, we prove nonnegativity of the probability density function. A crucial step in our analysis is to establish strong convergence of the sequence of Galerkin approximations by a combination of techniques, involving a novel compactness result for nonlocal PDEs. Lastly, we prove the uniqueness of weak solutions with sufficient regularity.
    }


    \maketitle

    \vspace{-.6cm}

	\input{1_intro}

    \input{2_prelim}
    \input{3_analysis}
    \input{4_outlook}

    \backmatter
    
    \renewcommand*{\bibfont}{\normalfont\scriptsize}
    \setlength{\bibsep}{2pt}
    \bibliography{literature.bib}
	
\end{document}

%% file: preamble.tex
\usepackage{microtype}

\allowdisplaybreaks
\usepackage{mathtools,enumitem,amsmath}
\usepackage[capitalise,nameinlink,noabbrev]{cleveref}
\usepackage{verbatim}
\hypersetup{breaklinks=true} 

\newtheorem{Theorem}{Theorem}
\newtheorem{Lemma}{Lemma}
\newtheorem{Definition}{Definition}
\newtheorem{Assumption}{Assumption}
\newtheorem{Remark}{Remark}

\crefname{subsection}{Subsection}{Subsections}
\crefformat{equation}{(#2#1#3)}
\crefformat{align}{(#2#1#3)}
\crefformat{enumi}{(#2#1#3)}
\crefname{figure}{Figure}{Figures}

\crefname{Theorem}{Theorem}{Theorems}
\crefname{Definition}{Definition}{Definitions}
\crefname{Lemma}{Lemma}{Lemmas}
\crefname{Assumption}{Assumption}{Assumptions}

\renewcommand{\rho}{\varrho}
\newcommand{\PC}{\mathcal{P}\mathcal{C}}

\newcommand{\psim}{\psi^m}

\newcommand{\pt}{\partial_t}
\newcommand{\eps}{\varepsilon}
\newcommand{\ds}{\,\textup{ds}}

\newcommand{\dx}{\,\textup{d}x}
\newcommand{\dxq}{\,\textup{d}(x,q)}

\newcommand{\rhol}{\rho^\l}

\renewcommand{\div}{\textup{div}}
\newcommand{\divq}{\div_{\!q}}
\newcommand{\dd}{\mathop{}\!\mathrm{d}}

\newcommand{\ddt}{\frac{\dd}{\dd\mathrm{t}}}
\newcommand{\dt}{\,\textup{d}t}
\renewcommand{\p}{\partial}

\newcommand{\con}{\subset}

\newcommand{\dq}{\,\textup{d}q}
\newcommand{\hpsi}{{\widehat{\psi}}}

\renewcommand{\l}{\ell}

\newcommand{\hpsil}{\hpsi^\l}

\newcommand{\hpsilm}{\hpsi^{\l,m}}

\newcommand{\ul}{u^\l}
\newcommand{\um}{u^{m}}
\newcommand{\ulm}{u^{\l,m}}
\newcommand{\rhom}{\rho^{m}}
\newcommand{\hpsilmn}{\hpsi^{\ell,m,n}}

\newcommand{\ulmn}{u^{\ell,m,n}}
\newcommand{\umn}{u^{m,n}}
\newcommand{\hpsimn}{\hpsi^{m,n}}

\newcommand{\hpsim}{\hpsi^{m}}

\newcommand{\nablaxq}{\nabla_{\!x,q}}
\newcommand{\nablaq}{\nabla_{\!q}}
\newcommand{\nablaqj}{\nabla_{\!q^j}}
\renewcommand{\R}{\mathbb{R}}
\newcommand{\N}{\mathbb{N}}
\newcommand{\I}{0,T}
\newcommand{\D}{D}
\newcommand{\W}{\mathcal{W}}
\renewcommand{\O}{\mathcal{O}}

\newcommand{|}{\vert}
\makeatletter  
\renewcommand\@biblabel[1]{#1.} 
\makeatother
\usepackage[misc]{ifsym}
\mathcode`\;="603B
\mathcode`\,="613B
\mathcode`\~="7E

\newcommand{\dhookrightarrow}{\hookrightarrow\hspace{-3mm}\hookrightarrow }

\allowdisplaybreaks

\makeatletter
\newcommand{\doi}[1]{}
\makeatother

%% file: 1_intro.tex
\section{Introduction}
Kinetic models of dilute polymeric fluids are extensively described in the monographs \cite{bird1987dynamics,ottinger2012stochastic} and the review article \cite{li2007mathematical}.
Such models involve the coupling of the macroscopic Navier--Stokes equations for the description of incompressible fluid flow and the Fokker--Planck equation for the mezoscopic processes associated with
the statistical properties of polymer molecules immersed in the fluid. 

Differential equations exhibiting nonlocal-in-time features have been the focus of considerable attention in the mathematical and engineering literature in recent years; see the textbooks on nonlocal effects with applications in viscoelasticity \cite{mainardi2022fractional,yang2020general}, mathematical finance \cite{fallahgoul2016fractional}, and mechanical processes \cite{atanackovic2014fractional,pilipovic2014fractional}. Such equations have an innate history effect and often involve fractional derivatives. They are of relevance in applications where memory effects are present and hereditary properties of materials are studied. The time-fractional Fokker--Planck system, in particular, models subdiffusive behaviour and has been previously studied in \cite{angstmann2015generalized,heinsalu2007use,magdziarz2008equivalence,Weron2008ModelingEquation,fritz2024well}. Further developments within the fractional calculus framework can be found in the survey papers \cite{HahnUmarov2011,Sandev2015,AwadMetzler2020}, where generalised memory kernels and fractional Fokker--Planck--Kolmogorov equations were analysed.
 
This paper is concerned with the existence and uniqueness of weak solutions to a system of nonlocal-in-time PDEs, whose local counterpart arises in the kinetic theory of dilute solutions of polymeric fluids. 
In addition to the classical theory, we incorporate memory/history effects and subdiffusive dynamics into the system. In our previous work~\cite{fritz2024analysis}, we considered a Riemann--Liouville time-fractional derivative in the system. Although that model captured essential physical phenomena, our analysis relied on a simplified, corotational, drag term in the Fokker--Planck equation, enabling us to establish energy estimates in a Hilbert space setting and prove the existence of weak solutions.
The extension of that analysis to a physically realistic drag term involves significant mathematical obstacles, primarily because of the behaviour of the Riemann--Liouville derivative at $t=0$. The fundamental difficulty lies in the singular nature, at initial time,  of solutions to the time-fractional Fokker--Planck equation involving the Riemann--Liouville time derivative, which conflicts with the physical requirement that the nonnegativity of the solution as a probability density function should be inherited from the nonnegativity of the initial datum specified at $t=0$. To address these limitations while preserving the memory effect, we introduce a nonlocal derivative involving a singular kernel of Prabhakar--Caputo type ($\mathcal{P}\mathcal{C}$ type) in the two diffusion terms and in the advection terms that appear in the Fokker--Planck equation. 
This approach is, in general, \textit{not} equivalent to the one we considered in our previous work \cite{fritz2024analysis}.  

The paper is structured as follows. In \Cref{Sec:Prelim}, we introduce the notation used throughout the paper and formulate the mathematical model that we shall study. We then define several function spaces of Sobolev type and recall some important results from the theory of fractional derivatives, including Alikhanov's inequality as a replacement for the product/chain rule. We close the section with the derivation of a novel compactness result involving nonlocal derivatives, which plays a central role in the mathematical analysis of the model. In \Cref{Sec:Analysis}, we state the weak formulation of the model problem and the assumptions concerning the data for the problem. The main results of the paper are contained in Theorems \ref{Thm:Main}, \ref{Thm:Main2} and \ref{Thm:WeakStrong}; they guarantee, respectively, the existence of a large-data global-in-time weak solution to the problem in terms of the velocity field and the probability density function, the existence of a pressure, and the uniqueness of weak solutions with sufficient regularity. The proof of Theorem \ref{Thm:Main} is based on a spatial Galerkin approximation in conjunction with a compactness argument. We also prove the nonnegativity of the probability density function, which enables us to use a suitable logarithmic relative entropy involving the Maxwellian associated with the model to derive an energy inequality. 
We close the paper in \cref{Sec:Outlook} with concluding remarks and formulate some open questions. We note in passing that spatial discretizations of dilute polymer models based on spatial finite element approximations were previously considered in \cite{barrett2009numerical,barrett2011finite,barrett2012finite} and of nonlocal-in-time nonlinear PDEs in \cite{fritz2021subdiffusive,fritz2022time,fritz2022equivalence}.

%% file: 2_prelim.tex
\section{Preliminaries} 
\label{Sec:Prelim}
In this section, we state our model problem and introduce the function spaces required in our analysis. Most of the analytical tools that we use can be found in standard textbooks such as \cite{roubicek2013nonlinear,Boyer2013,Fonseca2006,brezis2011functional}. Nonlocal derivatives and their properties play a key role in our analysis; standard bibliographical sources on the analysis of time-fractional differential equations include \cite{kubica2020time,jin2021fractional,diethelm2010analysis}. 

\subsection{Definitions}
Let $\Omega \subset \mathbb{R}^d$, $d \in \{2,3\}$, be a bounded Lipschitz domain, $T > 0$ the length of the time interval, and $\Omega_T := \Omega \times (\I)$ the space-time domain. The conformation space $\D \subset \mathbb{R}^{d \times K}$ is a $K$-fold Cartesian product $\D := \D^1 \times \dots \times \D^K$, where each $\D^j$, $j \in \{1, \dots, K\}$, is an open bounded $d$-dimensional ball centred at the origin, and the extended domains are defined as $\mathcal{O} := \Omega \times \D$ and $\mathcal{O}_T := \mathcal{O} \times (\I)$. Moreover, the Maxwellian distribution $M(q)=\prod_{j=1}^K M^j(q^j)$ on $D$ is given component-wise on $D^j$, $j\in\{1, \dots, K\}$, by:
\begin{equation} \label{Def:Maxwellian}
M^j(q^j)=\frac{e^{-U^j(\frac12 |q^j|^2)} }{\textstyle\int_{\D^j} e^{-U^j(\frac12|q^j|^2)} \dd q^j},
\end{equation}
for a prescribed spring potential $q^j \in D^j \mapsto U^j(\frac{1}{2}|q^j|^2)$ contained in $C^1(D^j;\mathbb{R}_{\geq 0})$, which is required to be such that $U^j(\frac{1}{2}|q^j|^2) \to + \infty$ as $\mbox{dist}(q^j, \partial D^j) \to 0$,  $j \in \{1, \ldots, K\}$.
Consequently, $M \in C^1_0(\overline{D})$.

Lastly, we introduce a kernel $k \in L^1_{\text{loc}}(\mathbb{R}_+)$  of $\mathcal{P}\mathcal{C}$ type; that is, we assume that $k$ is nonnegative and nonincreasing, and there exists a resolvent kernel $\tilde k \in L^1_{\text{loc}}(\mathbb{R}_+)$ such that the Sonine property $k * \tilde k = 1$ is fulfilled. In this case, we say that $(k,\tilde k)$ is a $\mathcal{P}\mathcal{C}$ pair and write $(k,\tilde k) \in \mathcal{P}\mathcal{C}$. We refer to \cite{zacher2008boundedness} for the analysis of a nonlocal diffusion equation with kernels of $\PC$ type. The prototypical example is the singular Abel kernel $g_\alpha(t) = t^{\alpha-1}/\Gamma(\alpha)$ for $\alpha \in (0,1)$, which is central to fractional calculus; we refer to \cite{Hanyga2020} for a discussion of regular versus singular kernels in generalised fractional derivatives. From now on $k$ will always be assumed to be a $\PC$ type kernel.

\subsection{Modelling}\label{Sec:Modelling}

Dilute polymers are viscoelastic fluids, which can be viewed as mixtures of a viscous base liquid, typically a Newtonian fluid, and elastic polymer macromolecules flowing in it (without self-interaction and with no interaction between the polymer molecules). 
This is reflected in the form of the Cauchy stress tensor, which is the sum of two parts: one part being the classical Newtonian viscous stress tensor, corresponding to the response of the viscous base liquid, and the other, 
called the elastic extra-stress tensor and denoted by $\mathbb{S}$, corresponding to the elastic response. In a bead-spring chain model for dilute polymers, see \cite{cascales1991simulation,kirkwood1967Macromolecules,barrett2011existence}, where each polymer chain immersed in the base fluid is assumed to consist of $K + 1$ beads linearly coupled with $K$ elastic springs to represent a polymer chain, the elastic extra-stress tensor $\mathbb{S}$ is defined by the Kramers expression (see eq.~\eqref{Kramers} below) through the probability density function $\psi=\psi(x,q,t)$, depending on $(x,t) \in Q_T$ and the conformation vector $q=((q^1)^{\rm T},\dots,(q^K)^{\rm T})^{\rm T} \in \R^{d \times K}$ of the chain, with the column vector $q^j=(q_1^j,\dots,q_d^j)^{\mathrm{T}}$ representing the $d$-component conformation vector of the $j$-th spring in the bead-spring chain. 

The evolution of the velocity $u$ and the pressure $p$ of the fluid is governed by the following system of partial differential equations that models the balance of linear momentum and incompressibility, respectively:
\begin{equation} \label{Def:NS2} \begin{aligned}
    \pt u + \div(u \otimes u) - \div(-p \mathbb{I} + 2\nu \mathbb{D}(u)  + \mathbb{S}(\hpsi)) &= 0 \quad \text{ in } \Omega_T, \\
    \div\, u &= 0 \quad \text{ in } \Omega_T,
\end{aligned} \end{equation} 
where $\nu$ is the viscosity coefficient and $\mathbb{D}(u):=\frac{1}{2}(\nabla u + (\nabla u)^{\mathrm{T}})$ is the symmetric velocity gradient. The density of the (incompressible) fluid under consideration is assumed to be constant and is, without loss of generality, taken to be identically equal to $1$. We define the polymeric extra-stress tensor $\mathbb{S}$ by the Kramers expression
\begin{equation}\label{Kramers}
\mathbb{S}(\hpsi(x,\cdot,t)):=- K \left(\int_D \psi(x,q,t) \dd q\right) \mathbb{I}  + \sum_{j=1}^K \int_\D \psi(x,q,t) q^j (q^j)^{\mathrm{T}} (U^j)'(\tfrac12 |q^j|^2) \dd q, 
\end{equation}
where $\mathbb{I} \in \mathbb{R}^{d \times d}$ is the identity matrix, and we recall that $U^j$ is the spring potential that appears in the definition of the Maxwellian \cref{Def:Maxwellian}. We observe that
$$M(q)\nabla_{q^j} (M(q))^{-1}=-(M(q))^{-1} \nabla_{q^j} M(q)=\nabla_{q^j} U^j(\tfrac12 |q^j|^2) = (U^j)'(\tfrac12|q^j|^2)q^j.$$
Inserting the equality $(U^j)'(\tfrac12|q^j|^2)q^j=-(M(q))^{-1} \nabla_{q^j} M(q)$
into the second summand in the definition of the extra-stress tensor \cref{Kramers}, performing partial integration using that $M|_{\partial D} = 0$, we deduce that
\begin{equation}\label{Def:TensorS}
\mathbb{S}(\hpsi)=\sum_{j=1}^K \int_\D M(q) \nabla_{q^j} \hpsi(x,q,t) \otimes q^j \dd q,
\end{equation}
%
where we introduced the Maxwellian-scaled probability density function $\hpsi:=\psi/M$. 

In this work, we assume that the evolution of $\psi$ is governed by the nonlocal-in-time Fokker--Planck equation
\begin{equation}\label{Eq:FokkerBeforeConv}\begin{aligned}&\pt \psi + \div(\pt(\tilde k*(u \psi)) + \sum_{j=1}^K \div_{q^j} (\pt(\tilde k*((\nabla u)q^j \psi)) \\[-.2cm] &\quad =  \eps\Delta \pt(\tilde k*\psi) +  \sum_{i=1}^K \sum_{j=1}^K \frac{A_{ij}}{4\lambda}\div_{q^i} (M \nabla_{q^j} \pt(\tilde k*\hpsi)),
\end{aligned}
\end{equation}
which exhibits anomalous diffusion effects and has transport coefficients that depend on the velocity field $u$. Here, $\eps>0$ denotes the centre-of-mass diffusion coefficient, $\lambda$ represents the elastic relaxation parameter, and $\mathbb{A}=(A_{ij})_{i,j=1}^K\in \R^{K\times K}_{\text{sym}}$ is a constant positive definite matrix. 
Permitting $\mathbb{A}$ to depend on $(x,q,t) \in \Omega \times D \times [0,T]$ would not create additional complications in our analysis, provided that the matrix-function $ \in L^\infty(\Omega \times D \times [0,T]; \mathbb{R}^{K \times K}_{\mathrm{sym}})$ is uniformly positive definite on $\Omega \times D \times [0,T]$. We shall therefore simply assume that $\mathbb{A} \in \mathbb{R}^{K \times K}_{\mathrm{sym}}$ is a constant positive definite matrix. We note that choosing $\tilde k =1$ in \eqref{Eq:FokkerBeforeConv} recovers the standard Navier--Stokes--Fokker--Planck system.

For simplicity of exposition, and without loss of generality, we set $\nu=1$, $\eps=1$, $\lambda=1/4$, $\mathbb{A}=\mathbb{I} \in \mathbb{R}^{K \times K}$, because none of our theoretical results depend on the specific values of these parameters. 
Additionally, we convolve \cref{Eq:FokkerBeforeConv} with the kernel function $k$; this then eliminates the nonlocal time derivatives from the spatial derivative terms in the equation thanks to the property $k*\pt(\tilde k * \psi)=\psi$; see \cref{FracProp} below. The first term becomes $k*\pt \psi=\pt(k*[\psi-\psi_0])$, which would represent the Caputo derivative of order $\alpha \in (0,1)$ if one chooses $k=g_{1-\alpha}$. Furthermore, to obtain an equation in terms of $\hpsi$,  we recall that $\psi = M \hpsi$ and note that the Maxwellian $M$ is independent of $t$ and $x$. We shall therefore study in this work the system
\begin{equation}\label{Def:System}\begin{aligned}
    \pt u + \div(u \otimes u) + \nabla p - \Delta u - \div\,  \mathbb{S}(\hpsi) &=0 &&\text{ in } \Omega_T, \\
    \div\, u &= 0 &&\text{ in } \Omega_T,
    \\\pt(Mk*[\hpsi-\hpsi_0]) + \div(M \hpsi u) + \divq(M\hpsi(\nabla u)q)&~ &&\\
    - \Delta (M\hpsi) -  \divq (M \nablaq\hpsi) &=0 &&\text{ in } \O_T,
\end{aligned}\end{equation}
which is supplemented by the following initial and boundary conditions
\begin{equation}\label{Def:Data}\begin{aligned}
u(0)&=u_0 &&\text{ in } \Omega, \\
    (k*[\hpsi-\hpsi_0])(0)&=0 &&\text{ in } \O, \\
    u \cdot n_{\p\Omega} &=0 &&\text{ on } \partial \Omega \times (\I), \\
    M\nabla\hpsi \cdot n_{\p\Omega} &=0 &&\text{ on }\p\Omega \times \D \times (\I), \\
    (M\hpsi(\nabla u)q^j-M\nablaq \hpsi)\cdot n_{\p D_j} &=0 &&\text{ on }\Omega \times \p \D_j \times (\I), 
\end{aligned}\end{equation}
where $n_{\partial \Omega}$ is the unit outward normal vector to $\partial \Omega$ and $n_{\partial D_j}$ is the unit outward normal vector to $\partial D_j$, $j=1,\ldots,K$. In \eqref{Def:System}$_1$ we have used \eqref{Def:System}$_2$ to replace $\div(2\mathbb{D}(u))$ appearing in \eqref{Def:NS2} with $\Delta u$ exploiting the identity $\div (\mathbb{D}(v)) = \frac{1}{2} \left( \nabla(\div \,v) + \Delta v\right)$.
\begin{Remark}
In nonlocal models, one generally does \textit{not} obtain $\hpsi \in C([\I]; X)$ for some function space $X$ consisting of functions defined on $\O$, and therefore the initial value is not attained in the sense that $\hpsi(t) \to \hpsi_0$ in $X$ as $t\to 0$. Instead, we shall prove that
 $k*[\hpsi-\hpsi_0] \in C([\I];X)$ and the initial condition is satisfied in the sense that $(k*[\hpsi-\hpsi_0])(t) \to 0$ in $X$ as $t \to 0$. 
 \end{Remark}

\subsection{Function spaces} 
We use standard notation for Lebesgue, Sobolev and Bochner spaces. For any pair of functions $u \in L^p(O)$ and $v\in L^{p'}(O)$ with $O \subset \R^m$ being a measurable set and $p' \in [1,\infty]$ being the H\"older conjugate of $p\in[1,\infty]$, we shall write
$$(u,v)_O:=\int_O u(z)v(z) \dd z.$$
Since we work with Maxwellian-weighted spaces we define, for any nonnegative $w \in C(\overline{O})$ and any $p \in [1,\infty)$, the weighted Lebesgue and Sobolev spaces
$$\begin{aligned}
L^p_w(O)&:=\big\{u \in L^p_{\text{loc}}(O): \|u\|_{L^p_w(O)}^p:=(w,|u|^p)_O < \infty \big\}, \\
W^{1,p}_w(O)&:=\big\{u \in W^{1,p}_{\text{loc}}(O): \|u\|_{W^{1,p}_w(O)}^p:=(w,|\nabla u|^p+|u|^p)_O < \infty \big\}.
\end{aligned}$$
We refer to \cite{kufner1984define,gol2009weighted} for the theory of weighted function spaces.
Of particular interest to us will be the case when $O = \mathcal{O} (= \Omega \times D)$ and $w = M$, the Maxwellian \cref{Def:Maxwellian} associated with the model under consideration. 

For $p^*:=pm/(m-p)$,  $1\leq p < m$, and $O \subset \mathbb{R}^m$ a Lipschitz domain, thanks to the Sobolev embedding theorem we have the continuous embedding $W^{1,p}(O) \hookrightarrow L^{p^*}(O)$, and by the Rellich--Kondrashov theorem the compact embedding $W^{1,p}(O) \dhookrightarrow L^{r}(O)$ for any $r \in [1,p^*)$.
Lastly, we define the following spaces of divergence-free $d$-component vector functions for $p \in [1,\infty)$:
$$\begin{aligned} 
W^{1,p}_{0,\div}(\Omega)&:=\big\{v \in W^{1,p}_0(\Omega)^d: \div\,  v = 0 \text{ in } \Omega\big\}, \\
L^p_{0,\div}(\Omega)&:=\big\{v \in L^p(\Omega)^d:  \div\,v=0 \text{ on } \Omega, \, v\cdot n_{\p\Omega} = 0 \text{ on } \p\Omega\big\}.
\end{aligned}$$

\subsection{Nonlocal derivative} 
The $L^p((\I))$ norm of a function $u:(\I) \to \R$ is bounded by its convolution with a $\PC$ kernel $k$, which can be seen from the following Hardy--Littlewood--Sobolev inequality (recall that $k$ is, by definition, nonnegative and nonincreasing): 
\begin{equation}\begin{aligned} \|u\|_{L^p(0,t)}^p := \int_0^t |u(s)|^p \ds  &\leq \frac{1}{k(t)} \int_0^t k(t-s) |u(s)|^p \ds \leq \frac{1}{k(T)} \big(k * |u|^p\big)(t)
\end{aligned} 
\label{Eq:KernelNorm}
\end{equation}
for all $t \in (0,T]$.

Next, we assume that $H$ is a Hilbert space and consider the space of functions $u : (\I) \to H$ with nonlocal derivatives with respect to $t \in (\I)$ defined by
$$
\begin{aligned}
\W^{p}_{0}(\I;H)&:=\big\{u \in L^p(\I;H) : k*u \in W^{1,p}_0(\I;H)\big\}, \quad p \in [1,\infty).
\end{aligned}$$
We note that the function $k*u \in W^{1,p}(\I;H)$ has a well-defined trace at $t=0$, even when $u$ itself does not have one, thanks to the continuous embedding
$$k*u \in W^{1,p}(\I;H) \hookrightarrow AC([0,T];H).$$
Furthermore, for given data $u_0 \in H$, we define the “translated” space
$$\begin{aligned}
\W^{p}_{u_0}(\I;H)&:=\big\{u \in L^p(\I;H) : u-u_0 \in \W^{p}_0(\I;H)\big\}.
\end{aligned}$$

The following \textit{inverse convolution property} follows from the Sonine property of the $\PC$ pairing $(k,\tilde k)$; see \cite[Lemma 3.2]{fritz2024analysis}:
\begin{equation} \label{FracProp} \begin{aligned}
 \tilde k*\pt(k*u)=u \qquad  \forall\, u \in \W^{p}_{u_0}(\I;H).  
 \end{aligned}\end{equation} 

It is well known that the classical chain rule does not hold for fractional derivatives. However,  the so-called \textit{fundamental identity} holds. It is stated in eq. \eqref{eq:fund} and
 can be seen as the analogue of the chain rule $[G(u)]' = G'(u)u'$; cf. \cite[Lemma 6.1]{Kemppainen2016a}. Suppose that $T>0$ and let $U$ be an open subset of $\mathbb{R}$. Let $k \in W^{1,1}((0,T))$, $G \in C^1(U)$, and $u \in L^1(0,T; U)$. Suppose further that the functions $G(u(\cdot))$, $G'(u(\cdot))u(\cdot)$ and $G'(u(\cdot))(k' \ast u)(\cdot)$ belong to $L^1((0,T))$ (which is the case if $u \in L^\infty((0,T))$); then, we have the following identity, referred to as the fundamental identity:
\begin{align}\label{eq:fund}
\begin{aligned}
G'(u(t)) \frac{\dd}{\dd t}(k \ast u)(t) = \frac{\dd}{\dd t}(k \ast G(u))(t)
+ \left(-G(u(t)) + G'(u(t))u(t)\right) k(t)\\
+ \int_0^t \left(G(u(t-s)) - G(u(t)) - G'(u(t))[u(t-s) - u(t)]\right) [-k'(s)]\dd s.
\end{aligned}
\end{align}
If, in addition, $u_0 \in U$, and $G$ is convex, then
$$G'(u(t)) \ddt (k*[u-u_0])(t)\geq \ddt (k*[G(u)-G(u_0)])(t)$$
for almost all $t \in (\I)$; cf. Lemma 6.2 in \cite{Kemppainen2016a}. 
It is straightforward to extend this result to functions with space-time dependence. Suppose that our $\PC$ type kernel $k$ is also in $W^{1,1}((0,T))$, $G \in C^1(\mathbb{R}_{\geq 0};\mathbb{R}_{\geq 0} )$ is convex, $u_0 \in L^1(\mathcal{O}; \mathbb{R}_{\geq 0})$, $G(u_0) \in L^1(\mathcal{O}; \mathbb{R}_{\geq 0})$
$u \in L^1(\mathcal{O}_T;\mathbb{R}_{\geq 0})$, and $G(u(\cdot))$, $G'(u(\cdot))u(\cdot)$, $G'(u(\cdot))(k' \ast u)(\cdot)$ and $G'(u(\cdot))u_0 k$ belong to $L^1(\mathcal{O}_T;\mathbb{R}_{\geq 0}))$; then,
\begin{equation}\label{Eq:Chain2}\int_{\mathcal{O}} G'(u(t)) \ddt (k*[u-u_0])(t)\dd q \dd x \geq \ddt \int_{\mathcal{O}}(k*[G(u)-G(u_0)])(t)\dd q \dd x
\end{equation}
for almost all $t \in (\I)$.


A result of this type for Hilbert-space-valued functions $u : (0,T) \to H$ is stated in \cite[Lemma 3.1]{fritz2021subdiffusive} for the case $k=g_\alpha$. We note that the proof there can be adapted to a general $\PC$-kernel pair $(k,\tilde k)$ as it does not require any results other than the Sonine property. 
Specifically, in the case of $G(s)=\frac12 s^2$, this inequality is referred to as Alikhanov's inequality, and has the form (see \cite[Theorem 2.1]{vergara2008lyapunov})
\begin{equation} \label{Eq:Chain}  
(u,\pt( k*[u-u_0])_H \geq \frac12 \pt(k*[\|u\|^2_H-\|u_0\|^2_H]) \quad \forall\, u \in \W^{2}_{u_0}(\I;H).
\end{equation}

\subsection{A generalisation of the Aubin--Lions lemma with nonlocal derivatives}
Since the equations under consideration are nonlinear, the existence of a weak solution obtained via a Galerkin approximation requires compactness arguments to ensure the strong convergence of the approximating sequences. To this end, let $(X,Y,Z)$ be a triple of Banach spaces such that $X \dhookrightarrow Y \hookrightarrow Z$.  It follows from the Aubin--Lions compactness lemma (see \cite[Theorem II.5.16]{Boyer2013}) that
\begin{equation} \label{Eq:aubinclassic}
\begin{aligned}  
L^p(0,T;X) \cap W^{1,1}(0,T;Z) &\dhookrightarrow L^p(0,T;Y), &&p\in [1,\infty), \\
 L^\infty(0,T;X) \cap W^{1,r}(0,T;Z) &\dhookrightarrow C([0,T];Y), &&r>1.
\end{aligned}
\end{equation} 
In the following, we will require a compactness result with weaker assumptions on the regularity of the time derivative. Specifically, we require the following lemma.
\begin{Lemma} \label{Lem:Kalita}
Let $X \dhookrightarrow Y \hookrightarrow Z$ be real Banach spaces, $p \in [1,\infty)$ and $T>0$. If $\mathcal{G}$ is a bounded subset of $L^p(0,T;X) \cap
BV(0,T;Z)$, then $\mathcal{G}$ is relatively compact in $L^p(0,T;Y)$.
\end{Lemma}
This result is a special case of  \cite[Proposition 2]{kalita2013convergence} with
$q=1$ there. Although the statement of Kalita’s Proposition~2 assumes
that $X$ is reflexive, a careful reading of the proof shows that this property is not actually used. We have, therefore, not assumed reflexivity in the above lemma.

When $Z$ is a dual Banach space, meaning that $Z$ admits a predual $Z'$ (that is, $Z = [Z']^*$), the space $BV(0,T;Z)$ can be identified with the space of functions $u \in L^1(0,T;Z)$ whose distributional time derivative $u_t$ is a $Z$-valued finite Radon measure on $(0,T]$, that is, $u_t \in \mathcal{M}(0,T;Z)$. The space $\mathcal{M}(0,T;Z)$ is endowed with the total variation norm
\[
\|v\|_{\mathcal{M}(0,T;Z)} := |v|(0,T],
\]
where $|v|$ is the total variation measure of $v$ defined for Borel sets $E\subseteq (0,T]$ by
\[
|v|(E) := \sup \left\{ \sum_{i=1}^n \|v(E_i)\|_Z \;:\;
E = \bigcup_{i=1}^n E_i,\; E_i \text{ are pairwise disjoint Borel sets} \right\}.
\]
In particular, when $Z$ has a predual, $Z'$, then $\mathcal{M}(0,T;Z)$ can be identified with the dual space $[C([0,T];Z')]^*$; see, e.g., \cite[Section~3.1]{ambrosio2000functions}. 
In our analysis, we will use \cref{Lem:Kalita} with functions $u \in L^1(0,T;Z)$ whose time derivative $u_t$ satisfies a bound of the form $|\langle u_t, \phi \rangle| \leq C \|\phi\|_{C([0,T];Z')}$; such a bound naturally implies that $u_t \in \mathcal{M}(0,T;Z)$, and therefore $u \in BV(0,T;Z)$.

There exist several nonlocal counterparts of the Aubin–Lions lemma in the literature, formulated under various assumptions on the spaces and kernels; see, for instance, \cite{Wittbold2020a,meliani2023unified,li2018some}. In particular, \cite[Theorem~3.2]{Wittbold2020a} establishes the compact embedding
\[
\big\{u \in L^p(0,T;X): \partial_t(k*(u-u_0)) \in L^p(0,T;Z) \big\} \dhookrightarrow L^p(0,T;Y), \quad p\in [1,\infty),
\]
for a Hilbert triple with $Z=X'$, and the proof readily extends to Banach spaces. We emphasize that, in this setting, the nonlocal (or fractional) time derivative must be bounded in an $L^p$-space with the same exponent $p$ as the bound for $u$. 
This requirement marks a key difference from the classical Aubin--Lions lemma, where an $L^1$-bound on the time derivative suffices.
In other words, compactness for the nonlocal case comes at the cost of stronger regularity in time, reflecting the weaker smoothing effect of the fractional derivative.

In what follows, we show, in analogy with \cref{Lem:Kalita}, that a bound of the nonlocal derivative in the measure space $\mathcal{M}(0,T;Z)$ is already sufficient to guarantee compactness. We note that the new assumption $\tilde k \in L^p(0,T)$ may impose a restriction on the fractional exponent in the case of Abel-type kernels. However, in our applications we will take $p=1$, in which case the condition $\tilde k \in L^1(0,T)$ is easily satisfied for Abel kernels.

\begin{Lemma}\label{thm:nonlocal_aubin_measures}
Let $X \dhookrightarrow Y \hookrightarrow Z$ be real Banach spaces where $Z$ is a dual Banach space.
Let $1 \le p < \infty$, $T > 0 $, and let $(k,\tilde k) \in \mathcal{PC}$ with $\tilde k \in L^p(0,T)$.
For $u_0 \in Z$, define
\[
\mathcal{F} = \left\{ u \in L^p(0,T;X) : \partial_t (k * (u - u_0)) \in \mathcal{M}(0,T;Z) \right\}\!.
\]
If $\mathcal{F}$ is bounded in $L^p(0,T;X)$ and
\[
M := \sup_{u \in \mathcal{F}} \left\| \partial_t (k * (u - u_0)) \right\|_{\mathcal{M}(0,T;Z)} < \infty,
\]
then $\mathcal{F}$ is relatively compact in $L^p(0,T;Y)$.
\end{Lemma}

\begin{proof}
For any $u \in \mathcal{F}$, define $v = \partial_t (k * (u - u_0)) \in \mathcal{M}(0,T;Z)$, and let $|v|$ be the total variation measure of $v$, so that $|v|(0,T] \le M$. 
By the definition of convolution with a Radon measure (see \cite[Definition 2.1]{ambrosio2000functions}) and the inverse convolution property \cref{FracProp}, we have
\[
u(t) - u_0 = (\tilde k * v)(t) = \int_0^t \tilde k(t-s) \, \dd v(s), \quad t \in [0,T],
\]
where the integral is understood as a Bochner integral with respect to the vector measure $v$.
We will show that $\mathcal{F}$ is relatively compact in $L^p(0,T;Y)$ by verifying the conditions of the Fr\'echet--Kolmogorov theorem \cite[Theorem~1]{simon1986compact}, which are the following:\\  
(i) for every $0 < t_1 < t_2 < T$, the set
\[
G(t_1,t_2) = \left\{ \int_{t_1}^{t_2} u(t) \, \mathrm{d}t : u \in \mathcal{F} \right\}
\]
is relatively compact in $Y$, and\\  
(ii) $\mathcal{F}$ is strongly integrally equicontinuous in $L^p(0,T;Y)$, i.e.,
\[
\lim_{h \to 0} \sup_{u \in \mathcal{F}} \int_0^{T-h} \|u(t+h) - u(t)\|_Y^p \, \mathrm{d}t = 0.
\]

\noindent\textit{Step (i).} For any $u \in \mathcal{F}$, we have
\[
\left\| \int_{t_1}^{t_2} u(t) \, \mathrm{d}t \right\|_X 
\le \int_{t_1}^{t_2} \|u(t)\|_X \, \mathrm{d}t 
\le (t_2 - t_1)^{1 - \frac{1}{p}} \|u\|_{L^p(0,T;X)}.
\]
Thus $G(t_1,t_2)$ is bounded in $X$, and the compact embedding $X \dhookrightarrow Y$ implies its relative compactness in $Y$. \smallskip

\noindent\textit{Step (ii).}
For $h > 0$, the difference $u(t+h) - u(t)$ decomposes as
\begin{equation}\label{Eq:Decomp}
\begin{aligned}
u(t+h) - u(t)
&= \int_0^t [\tilde k(t+h-s) - \tilde k(t-s)] \, \dd v(s) + \int_t^{t+h} \tilde k(t+h-s) \, \dd v(s) \\
&=:(\Delta_h \tilde k * v)(t)+B_h(t),
\end{aligned}
\end{equation}
where $\Delta_h \tilde k(t) := \tilde k(t+h) - \tilde k(t)$ (with $\tilde k(t) = 0$ for $t < 0$).
By Minkowski’s integral inequality and the definition of total variation, we can bound the first term  in \cref{Eq:Decomp} as
\begin{equation}\label{Eq:Decomp1}
\begin{aligned}
\|\Delta_h \tilde k * v\|_{L^p(0,T-h;Z)}
&\le \int_0^T \|\Delta_h \tilde k(\cdot - s)\|_{L^p(0,T)} \, \dd |v|(s)
\\ &= \|\Delta_h \tilde k\|_{L^p(0,T)} \, |v|(0,T)
\\ &\le M \|\Delta_h \tilde k\|_{L^p(0,T)} \xrightarrow[h\to0]{} 0.
\end{aligned}
\end{equation}
Next, we consider the second term in \cref{Eq:Decomp}. 
By Jensen’s inequality for integrals with respect to a finite measure,
$$
|B_h(t)|^p
\le |v|(t,t+h]^{\,p-1}\int_t^{t+h}|\tilde k(t+h-s)|^p\,\dd|v|(s).
$$
Integrating over $t\in(0,T-h)$ and using Tonelli’s theorem yields
$$
\begin{aligned}
\|B_h\|_{L^p(0,T-h;Z)}^p
&\le M^{p-1}\!\int_0^{T-h}\!\!\int_t^{t+h}|\tilde k(t+h-s)|^p\,\dd|v|(s)\,\dd t \\
&= M^{p-1}\!\int_0^T\!\int_{(s-h)_+}^{\min(s,T-h)}|\tilde k(t+h-s)|^p\,\dd t\,\dd|v|(s).
\end{aligned}
$$
For fixed $s$, performing the change of variables $r=t+h-s$ gives
$$
\int_{(s-h)_+}^{\min(s,T-h)}|\tilde k(t+h-s)|^p\,\dd t
\le \int_0^h |\tilde k(r)|^p\,\dd r.
$$
Hence
\begin{equation}\label{Eq:Decomp2}
\|B_h\|_{L^p(0,T-h;Z)}^p
\le M^{p-1}|v|(0,T)\int_0^h |\tilde k(r)|^p\,\dd r
\le M^p\|\tilde k\|_{L^p(0,h)}^p \xrightarrow[h\to0]{} 0.
\end{equation}
Combining the estimates \cref{Eq:Decomp1} and \cref{Eq:Decomp2}, and returning to \cref{Eq:Decomp}, we have
\[
\sup_{u \in \mathcal{F}} \int_0^{T-h} \|u(t+h) - u(t)\|_Z^p \, \mathrm{d}t
\le M^p \left( \|\Delta_h \tilde k\|_{L^p(0,T)}^p + \|\tilde k\|_{L^p(0,h)}^p \right)
\xrightarrow[h\to0]{} 0.
\]
Thus $\mathcal{F}$ is strongly integrally equicontinuous in $L^p(0,T;Z)$. Finally, applying the so called Ehrling lemma \cite[Lemma~7.6]{roubicek2013nonlinear} (whose proof, as has been pointed out in the recent paper \cite{Olsen-Holden} by Hanche-Olsen \& Holden, is in fact due to J.-L. Lions \cite[Prop. 4.1, p. 59]{MR153974} and \cite[Lemma 5.1, p. 58]{MR259693}; for this reason, Hanche-Olsen \& Holden suggest `Aubin--Lions lemma' instead of the misnomer `Ehrling's lemma') for any $\varepsilon > 0$ there exists a $C_\varepsilon > 0$ such that
\[
\|x\|_Y^p \le \varepsilon \|x\|_X^p + C_\varepsilon \|x\|_Z^p \quad \forall x \in X.
\]
Hence, for any $u \in \mathcal{F}$ and $h > 0$,
\begin{align*}
\int_0^{T-h} \|u(t+h) - u(t)\|_Y^p \, \mathrm{d}t
&\le \varepsilon \int_0^{T-h} \|u(t+h) - u(t)\|_X^p \, \mathrm{d}t \\
&\quad + C_\varepsilon \int_0^{T-h} \|u(t+h) - u(t)\|_Z^p \, \mathrm{d}t.
\end{align*}
The first term is bounded by $\varepsilon C$ because $\mathcal{F}$ is bounded in $L^p(0,T;X)$, while the second term tends to $0$ uniformly in $u \in \mathcal{F}$ as $h \to 0$. Therefore,
\[
\lim_{h \to 0} \sup_{u \in \mathcal{F}} \int_0^{T-h} \|u(t+h) - u(t)\|_Y^p \, \mathrm{d}t = 0.
\]
This proves the required strong integral equicontinuity in $L^p(0,T;Y)$.
By the Fr\'echet--Kolmogorov theorem, $\mathcal{F}$ is therefore relatively compact in $L^p(0,T;Y)$.
\end{proof}

%% file: 3_analysis.tex
\section{Analysis} \label{Sec:Analysis}

Before we state our main result on the existence of weak solutions to the nonlocal Navier--Stokes--Fokker--Planck system, we first specify the definition of a weak solution.

\begin{Definition} \label{Def:Weak}
    We call a tuple $(u,\hpsi)$ a weak solution to the nonlocal Navier--Stokes--Fokker--Planck system \eqref{Def:System}, \eqref{Def:Data} provided that 
    $$\begin{aligned}
        u &\in C_w([\I];L^2_{0,\div  }(\Omega)) \cap L^2(\I;W^{1,2}_{0,\div  }(\Omega)) \cap W^{1,4/d}(\I;[W^{1,2}_{0,\div  }(\Omega)]^*), \\
        \mathbb{S}(\hpsi) &\in L^2(\I;L^2(\Omega)^{d\times d}), \\
        \hpsi &\in L^\infty(\Omega_T;L^1_M(\D))  
        ~  \text{ with } \hpsi \geq 0 \text{ a.e. in } \O_T, \\
        \nabla \hpsi &\in L^2(\Omega_T;L^{1}_M(\D)^{d(K+1)}), \\
        Mk*\hpsi &\in L^1(0,T;L^1(\mathcal{O}))\cap BV([\I];[W^{1,\infty}(\O)]^*),\\
        \hpsi \ln\hpsi &\in L^1_M(\O_T),
    \end{aligned}$$
    satisfying
    \begin{equation}\label{Def:SysVar}\begin{aligned}
        &\langle \pt u,w \rangle -(u \otimes u,\nabla w)_{\Omega_T} \\
        &\quad +  (\nabla u,\nabla w)_{\Omega_T}  +(\mathbb{S}(\hpsi),\nabla w)_{\Omega_T}-(f,w)_{\Omega_T}=0, \\
        & \langle \pt(Mk*[\hpsi-\hpsi_0]),\zeta \rangle + (M \nabla \hpsi,\nabla\zeta)_{\O_T}+ (M\nablaq \hpsi,\nablaq \zeta)_{\O_T}& \\    
        &\quad - (Mu\hpsi,\nabla\zeta)_{\O_T}-  (M\hpsi (\nabla u) q,\nablaq \zeta)_{\O_T} =0, 
    \end{aligned}\end{equation}
    with $\mathbb{S}(\hpsi) = \sum_{j=1}^K \int_\D M(q) \nablaqj \hpsi \otimes q_j \dq$,
    for any $w \in L^{4/(4-d)}(\I;W^{1,2}_{0,\div  }(\Omega))$ and $\zeta \in C ([\I];W^{1,\infty}(\O))$, and the initial data is attained in the sense 
    $$\lim_{t \to 0} (u(t)-u_0,w)_\Omega + (M(k*[\hpsi-\hpsi_0])(t),\zeta)_\O=0 \quad \forall w \in L^2_{0,\div  }(\Omega), \zeta \in L^\infty(\O).$$
\end{Definition}

We make the following assumptions on the model parameters and the functions that we require in our proof of the existence of weak solutions.

\begin{Assumption} \label{Assumptions}
    We assume the following:
    \begin{itemize} \itemsep.2em
        \item $K \in \N$ arbitrary, $\D^j \subset \R^d$, $d \in \{2,3\}$ and $j=1,\dots,K$, bounded open balls centered at the origin;
        \item $\Omega \subset \R^d$ bounded Lipschitz domain;
        \item $(k,\tilde k) \in \PC$ with $k,\tilde k \in L^1(\I)$;
        \item $f \in L^2(\I;L^2(\Omega)^d)$;
        \item $M \in C^1_0(\overline{D})$
        with $M > 0$ on $D$; 
        \item $u_0 \in L_{0,\div  }^2(\Omega)$;
        \item $\hpsi_0 \geq 0$ a.e. in $\O$ with $\hpsi_0 \ln \hpsi_0 \in L^1_M(\O)$ and $\rho_0 :=\int_\D M(q) \hpsi_0(\cdot,q) \dq \in L^\infty(\Omega)$.
    \end{itemize}
\end{Assumption}

The main result of this work regarding the existence of weak solutions to the nonlocal model \eqref{Def:System} is the following. 

\begin{Theorem} \label{Thm:Main}
    Let Assumption \ref{Assumptions} hold. Then, there exists at least one weak solution $(u,\hpsi)$ to the system \eqref{Def:System}, \eqref{Def:Data} in the sense of \Cref{Def:Weak}. 
\end{Theorem}

Before we begin to prove the theorem, we first comment on the existence of a pressure. We consider divergence-free test functions in the variational form for the Navier--Stokes equations, and thus the pressure is eliminated from the variational form. However, after having proved the existence of a solution $(u,\hpsi)$ as stated in \cref{Thm:Main}, we can associate a pressure with the velocity; see \cite[Ch.~V.1.5]{Boyer2013}. 

\begin{Theorem}\label{Thm:Main2}
    There exists a pressure $p \in W^{-1,\infty}(0,T;L_0^2(\Omega))$, such that $(u,p,\hpsi)$ solves the nonlocal model \eqref{Def:System} in distributional sense.
\end{Theorem}

\begin{proof}
\cref{Thm:Main} guarantees the existence of a weak solution $(u,\hpsi)$. We then define
$q:=\pt u + \div  (u \otimes u) - \nu\Delta u - \div\, \mathbb{S}(\hpsi)-f,$
and from the regularity of the solution, as stated in \cref{Def:Weak}, it directly follows that $$q \in W^{-1,\infty}(\I;W^{-1,2}(\Omega)^d):=\mathcal{L}(W_0^{1,1}(\I);W^{-1,2}(\Omega)^d).$$ Further, for any $\eta \in W_0^{1,1}(\I)$ and $w \in W^{1,2}_{0,\div  }(\Omega)$, we have that
$$\begin{aligned}
&\langle q(\eta),w \rangle_{W^{-1,2}(\Omega) \times W^{1,2}_0(\Omega)} \\ &=\bigg\langle \int_0^T \big( -u\eta' + \div  (u \otimes u) \eta - \nu\Delta u \eta - \div\, \mathbb{S}(\hpsi) \eta -f\eta \big) \dt,w \bigg\rangle_{W^{-1,2}(\Omega) \times W^{1,2}_0(\Omega)}\\
&=\int_0^T \Big(\ddt (u,w)_{\Omega} - (u \otimes u,\nabla w)_\Omega  + (\nabla u,\nabla w) + (\mathbb{S}(\hpsi),\nabla w)_\Omega\\
& \qquad - (f,w)_{\Omega}  \Big) \eta(t) \dt \\
&=0.
\end{aligned}$$ By the de Rham lemma, see \cite[Theorem IV.2.3]{Boyer2013}, there exists a unique $p \in W^{-1,\infty}(\I;L_0^2(\Omega))$ such that $\nabla p =-q$.
\end{proof}

Further, we prove a uniqueness result for weak solutions to the nonlocal micro-macro model in the sense of \cref{Def:Weak} that further fulfil some regularity assumptions. The assumed regularity aligns with the Escauriaza--Seregin--{\v{S}}ver\'{a}k condition \cite{escauriaza2003l_3} for the backward uniqueness for the Navier--Stokes equations, ensuring smoothness of strong solutions. Furthermore, we require that the radius of each $D_j$, $j \in \{1,\dots,K\}$, in the configuration space is sufficiently small. This is a justified assumption for the microscopic setting. However, we note that the smallness of the radii can be heavily relaxed if a small final time $T$ is assumed. In fact, we will see in the proof that the radii will be subtracted from $\tilde k(T)$, and $\tilde k$ is nonincreasing with $\lim_{t \to 0}\tilde k(t)=\infty$. 
\begin{Theorem} \label{Thm:WeakStrong}
    Let Assumption \ref{Assumptions} hold and further let the volume of $D$  
    be sufficiently small. Every weak solution $(u,\hpsi)$ with regularity $u\in L^\infty(0,T;W^{1,3}_0(\Omega)^d)$ and $\hpsi \in L^\infty(0,T;W^{1,\infty}(\O))$ is unique.
\end{Theorem}

As the proofs of \cref{Thm:Main} and \cref{Thm:WeakStrong} on the well-posedness of the model \cref{Def:System} are quite technical, we first provide a structure with five sub-steps.

\smallskip

\noindent\textbf{Step 1 ($\l,m,n$):} We start by introducing an $\l$-approximate problem based on truncating $\hpsi$ in the drag term of the nonlocal Fokker--Planck equation. Additionally,  Kramers' expression is truncated to obtain an energy balance. To this truncated problem with solution $(u^\l, \hpsi^\l)$, we then apply a Galerkin--Faedo approximation, yielding  a discrete solution $(\ulmn, \hpsilmn)$ of the truncated system, where
$n \in \mathbb{N}$ stands for the (finite) dimension of the approximate divergence-free velocity subspace, and $m \in \mathbb{N}$ for the (finite) dimension of the probability density subspace.
\medskip

\noindent\textbf{Step 2 ($n \to \infty$):} In this step, we derive $n$-independent a priori estimates allowing us to pass to the limit $n \to \infty$ in the $(n,m)$-th Galerkin system, i.e., we obtain a semi-discrete solution $(u^{\l,m}, \hpsilm)$. Of crucial importance for the following steps is the nonnegativity  of $\hpsilm$, which is also established in this step.  It is derived by using a technique based on the weak maximum principle for nonlocal diffusion equations.
\medskip

\noindent\textbf{Step 3 ($m \to \infty$):} We derive $m$-independent a priori bounds of the sequence $(\hpsilm)_m$, allowing us to pass to the limit $m \to \infty$ in the $m$-th Galerkin system. This argument involves proving the strong convergence of $\hpsilm$ to $\hpsil$ in the Maxwellian-weighted space $L^1_M(\O_T)$. 
We show that $(G(\hpsilm))_m$ is bounded in $L^\infty(0, T; L_M^1(\O))$, where $G(s) = s \ln s + e^{-1}$, which we use in combination with our novel nonlocal compactness result to first obtain weak convergence and then strong convergence to $\hpsil$ in $L^1_M(\O_T)$.
\medskip

\noindent\textbf{Step 4 ($\l \to \infty$):} Finally, we collect the necessary weak and strong convergence results from all our energy estimates. We pass to the limit $\l \to \infty$ in the truncation parameter $\l$ within the various approximating sequences and then pass to the limit $\l \to \infty$ in Kramers' expression, which then completes the proof of \cref{Thm:Main}. \medskip

\noindent\textbf{Step 5 (Uniqueness):} In this last step, we prove \cref{Thm:WeakStrong}, that is, the uniqueness of weak solutions with sufficient regularity. We do this by means of a suitable testing procedure.

\subsection{Step 1: Truncation and Galerkin approximation}
Following the ideas in \cite{bulicek2013existence}, we truncate the extra-stress tensor $\mathbb{S}$, see \eqref{Def:TensorS}, in the Navier--Stokes equation using a cutoff function $\Gamma \in C_c^\infty(-2,2)$ such that $\Gamma(s)=1$ for all $s \in [-1,1]$. In addition, we define $\Gamma_\l(s)=\Gamma(s/\l)$ for an arbitrary $\l \in \N$. The primitive and a scaled version of $\Gamma_\l$ are denoted by, respectively, 
\begin{equation}\label{Def:Tl} 
\begin{aligned}
T_\l(s)&=(1*\Gamma_\l)(s) \quad \text{and} \quad \Lambda_\l(s) =s\Gamma_\l(s).
\end{aligned}\end{equation}
For a real-valued function $\Phi$ defined on $D$, we define the $\l$-th approximation of $\mathbb{S}$ as follows
\begin{align*}
\mathbb{S}_\l(\Phi):=\mathbb{S}(T_\l(\Phi)) .
\end{align*}
The definition of $\mathbb{S}$ given in \eqref{Def:TensorS} and integration by parts, using that $M \in C^1_0(D)$, yield
\begin{align}
\mathbb{S}_\l(\Phi) 
    &=\int_\D \bigg[ - KM(q)T_\l(\Phi) \mathbb{I}+\sum_{j=1}^K T_\l(\Phi)\nablaqj M(q) \otimes q^j  \bigg] \dq,
\label{sl} \end{align}
 see \cite[Lemma 3.1]{barrett2011existence} for details. Then, we define the $\l$-approximation $(u^\l, p^\l, \hpsil)$ of the nonlocal Navier--Stokes--Fokker--Planck system \cref{Def:System} as
\begin{align} 
\begin{aligned}
\label{Def:Truncation_Velocity}
\pt u^\l + \div  (\Gamma_\l(|u^\l|^2)u^\l \otimes u^\l) -  \nabla u^\l +\nabla p^\l - \div  \, \mathbb{S}_\l(\hpsil) &=  f &&\text{in } \Omega_T, \\
\div  \, u^\l &=0 &&\text{in } \Omega_T,
\end{aligned}
\end{align}
coupled with
\begin{align}
\label{Def:Truncation_PDF} M\pt  (k*[\hpsil-\hpsil_0]) - \Delta  (M\hpsil)-\div  _{\!q}(M\nablaq \hpsil)+ \div  (M\hpsil \ul) &\notag \\  +~\div  _{\!q} (M\Lambda_\l(\hpsil)(\nabla \ul) q) &=0 &&\text{in } \O_T,
\end{align}
with initial and boundary data as before in \cref{Def:Data}, with $u$ and $\hpsi$ being replaced by $u^\l$ and $\hpsil$. In order to avoid technical difficulties, we truncate the initial condition of $\hpsil$ as follows: $(M(k*[\hpsil-\hpsil_0])(0),\zeta)_\O=0$ for any $\zeta \in L^\infty(\O)$ with $\hpsil_0:=T_\l(\hpsi_0)$. We note that we also modified the nonlocal Fokker--Planck equation in the drag term by replacing $\hpsil$ by $\Lambda_\l(\hpsil)$ to preserve the energy identity following the truncation process resulting in \eqref{Def:Truncation_Velocity}.

We begin by stating the result concerning the existence of a global-in-time large-data weak solution to the $\l$-truncated system \eqref{Def:Truncation_Velocity}, \eqref{Def:Truncation_PDF}, which will be proved at the end of Step 3. In Step 4 we shall then pass to the limit $\ell \to \infty$ to remove the truncation. 

\begin{Lemma} \label{Lem:ExistenceLapprox}
    Let Assumption \ref{Assumptions} hold. Then, there exists a weak solution $(u^\l,\hpsil)$ to the $\l$-approximation \eqref{Def:Truncation_Velocity}, \eqref{Def:Truncation_PDF} such that
    $$\begin{aligned}
        u^\l &\in L^\infty(\I;L^2_{0,\div  }(\Omega)) \cap L^2(\I;W^{1,2}_{0,\div  }(\Omega)) \cap W^{1,2}(\I;[W^{1,2}_{0,\div  }(\Omega)]^*), \\
        \mathbb{S}_\l(\hpsil) &\in L^\infty(\I;L^\infty(\Omega)^{d\times d}), \\
        \hpsil &\in L^\infty(\Omega_T;L^1_M(\D)) \text{ with }\hpsil \geq 0 \text{ a.e. in } \O_T, \\
        \nablaxq \hpsil &\in L^2(\Omega_T;L^1_M(\D)^{d(K+1)}), \\
        Mk*\hpsil &\in L^1(0,T;L^1(\mathcal{O}))\cap BV(\I;[W^{1,\infty}(\O)]^*),
    \end{aligned}$$
    with
    \begin{equation} \label{Eq:Lapproximation}\begin{aligned}
        &\langle \pt \ul,w \rangle -(\Gamma_\l(|\ul|^2)\ul \otimes \ul,\nabla w)_{\Omega_T}\\
        &\quad + (\nabla \ul,\nabla w)_{\Omega_T}+(\mathbb{S}_\l(\hpsil),\nabla w)_{\Omega_T}   =(f,w)_{\Omega_T}, \\[.2cm]
        & \langle M\pt(k*[\hpsil-\hpsil_0]),\zeta \rangle   + (M\nabla\hpsil,\nabla\zeta)_{\O_T}   +(M\nablaq \hpsil,\nablaq \zeta)_{\O_T} \\ 
        &\quad =  (M\ul\hpsil,\nabla\zeta)_{\O_T}+(M\Lambda_\l(\hpsi) (\nabla \ul) q ,\nablaq \zeta)_{\O_T},
    \end{aligned}\end{equation}
    for any $w \in L^2(\I;W^{1,2}_{0,\div  }(\Omega))$ and $\zeta \in C([\I]; W^{1,\infty}(\O))$, and the initial data is attained in the sense
    $$\lim_{t \to 0} (\ul(t)-u_0,w)_\Omega + (M(k*[\hpsil-\hpsil_0])(t),\zeta)_\O=0 \quad \forall w \in L^2_{0,\div  }(\Omega),\, \zeta \in L^\infty(\O).$$
    Moreover, the tuple $(\ul,\hpsil)$ satisfies the energy estimate
    \begin{equation} \label{Eq:EnergyLexistenc} \begin{aligned} & \|\ul\|_{L^\infty(0,T;L^2(\Omega))}^2 +  \|\nabla \ul\|_{L^2(\Omega_T)}^2+ \|G(\hpsil)\|_{L^1_{M}(\O_T)} + \Big\|\nablaxq \sqrt{\hpsil} \Big\|_{L^2_{M}(\O_T)}^2 \\[0cm] &\quad \leq C\Big( \|\ul_0\|_{L^2(\Omega)}^2 + \|G(\hpsil_0)\|_{L^1_{M}(\O)}+\|f\|_{L^{2}(\Omega_T)}^{2} \Big) \\
    &\quad \leq C(\l).
    \end{aligned}\end{equation}
\end{Lemma}

We point out that Lemma \ref{Lem:ExistenceLapprox} is formulated on a divergence-free velocity space, and thus the pressure is absent. Although \eqref{Def:Truncation_Velocity} also involves the pressure $p^\ell$ acting as a Lagrange multiplier, it does not play a role in the Fokker--Planck equation; this motivates us to consider the weak formulation \eqref{Eq:Lapproximation}.

For simplicity of notation, we temporarily omit the superscript $\l$ in $\hpsil$ and $\ul$ and simply write $\hpsi$ and $u$. Analogously, 
in terms of our abbreviated notation $(\ulmn,\hpsilmn)$ reads $(\umn, \hpsimn)$.
The omitted superscript $\ell$ will be reinstated later, in Step 4, when we consider passing to the limit $\l \to \infty$. We emphasize that by doing so we will arrive at upper bounds where the index $\l$ appears explicitly on the right-hand sides of inequalities, but on the left-hand sides of inequalities it is suppressed in our notation. However, it is of utmost importance to carefully track all $\l$-dependencies in the constants because in Step 4 we have to pass to the limit $\l \to \infty$, and thus uniformity in $\l$ of our bounds has to be shown. 

To prove \cref{Lem:ExistenceLapprox}, we introduce a Galerkin approximation of the $\l$-approximated system \eqref{Eq:Lapproximation}. 
To this end we define an approximated Maxwellian $M^m$, $m \in \N$, by 
\begin{equation} \label{Def:ApproxMax}
M^m:= M
+ 1/m.
\end{equation}
We choose the same Galerkin subspaces as in \cite{bulicek2013existence}. 
In fact, there exists a countable set $\{w_i\}_{i \in \N}$ in $W^{1,2}_{0,\div  }(\Omega) \cap W^{d+1,2}(\Omega)^d$, whose linear span is dense in $L^2_{0,\div  }(\Omega)$, such that $w_i$, $i \in \N$, are orthogonal in the inner product of $W^{d+1,2}(\Omega)^d$ and orthonormal in the inner product of $L^2(\Omega)^d$. Similarly, there is a countable set $\{\zeta_i^m\}_{i \in \N}$ of eigenfunctions in $W^{1,2}(\O)$ that are orthogonal in $W^{1,2}_{M^m}(\O)$ and orthonormal in $L^2_{M^m}(\O)$. Finally, we fix $m,n \in \N$ and look for a tuple $(u^{m,n},\hpsi^{m,n})$ given by
$$\begin{aligned} u^{m,n}(x,t)&:=\sum_{i=1}^m c_i^{m,n}(t) w_i(x),\qquad \hpsi^{m,n}(x,q,t):=\sum_{i=1}^n d_i^{m,n}(t) \zeta_i^m(x,q),
\end{aligned}$$
that solves the Galerkin system
    \begin{equation} \label{Eq:Galerkin}\begin{aligned}
        &\langle \pt \umn,w_i \rangle  -(\Gamma_\l(|\umn|^2)\umn \otimes \umn,\nabla w_i)_\Omega + (\nabla \umn,\nabla w_i)_\Omega +(\mathbb{S}_\l(\hpsimn),\nabla w_i)_\Omega \\ &\quad = (f,w_i)_{\Omega} \\[.2cm]
        & \langle M^m\pt( k*[\hpsimn-\hpsimn_0]),\zeta^m_i \rangle  + (M^m\nabla\hpsimn,\nabla\zeta^m_i)_{\O}+(M^m\nablaq \hpsimn,\nablaq \zeta^m_i)_{\O} \\ & \quad = (M^m\umn\hpsimn,\nabla\zeta^m_i)_{\O} +({M} \Lambda_\l(\hpsimn)(\nabla \umn) q,\nablaq \zeta^m_i)_{\O}, 
    \end{aligned}\end{equation}
    for a.e. $t\in (\I)$, with initial conditions
\begin{equation} \label{Eq:DefinitionInitials}
    \begin{aligned}\umn(0)&=u_0^m:=\sum_{i=1}^m (u_0,w_i)_\Omega w_i, \\
    (k*[\hpsimn-\hpsimn_0])(0)&=0,      \end{aligned}\end{equation}
    where we define $\hpsimn_0$ by
    \begin{equation} \label{Eq:DefinitionInitials2}\hpsi_0^{m,n}:=\sum_{i=1}^n (T_\l(\hpsi_0^m),\zeta_i^m)_\O \zeta_i^m,\end{equation} with $\hpsi_0^m:=\hpsi_0 M/M^m$.
We stress that in \eqref{Eq:Galerkin} the original Maxwellian $M$ is used in combination with $\Lambda_\l (\cdot)$, while at all other locations the approximated one $M^m$ is present.
    
     The local-in-time existence of $(\umn,\hpsimn)$ for fixed $m,n$ follows from the standard theory of nonlocal differential equations, see \cite[Ch.~2, Theorem 3.5]{gripenberg1990volterra}. The estimates established below allow us to extend the solution until the fixed final time $T$.

\subsection{Step 2: Uniform a priori estimates in $n$ and the limit process $n \to \infty$}
First, we state and prove a lemma that provides us with a $n$-independent energy estimate for the tuple $(\umn,\hpsimn)$. Thereafter, we pass to the limit $n \to \infty$ in the Galerkin system \cref{Eq:Galerkin}, \cref{Eq:DefinitionInitials}.

     \begin{Lemma} \label{Lem:Nindep} Let Assumption \ref{Assumptions} hold. For the Galerkin solution $(\umn,\hpsimn)$ of \cref{Eq:Galerkin}, the following $n$-independent a priori estimate holds:
     \begin{equation} \label{Eq:NindepPsi}\begin{aligned} & \|\umn\|_{L^\infty(\Omega_T)}^2 + \|\nabla \umn\|_{L^\infty(\Omega_T)}^2+\|\hpsimn\|_{L^\infty(\I;L^2(\O))}^2+ \|\nablaxq \hpsimn\|_{L^2(\O_T)}^2   \\ &\quad \leq C(m,\l,T).
         \end{aligned}\end{equation}
     \end{Lemma}
     \begin{proof}
         Thanks to our assumptions on $M$, $M^m$ and $T_\l$, see Assumption \ref{Assumptions}, \cref{Def:Tl} and \cref{Def:ApproxMax}, we have that 
         $$\|\mathbb{S}_\l(\hpsimn)\|_{L^\infty(0,T;L^\infty(\Omega))}\leq C\l.$$ 
         We multiply the $i$-th equation in \eqref{Eq:Galerkin}$_1$ by $c_i^{m,n}(t)$ and sum with respect to $i=1,\dots,m$ to deduce that 
         $$\frac12 \ddt \|\umn\|_{L^2(\Omega)}^2 + \|\nabla \umn\|_{L^2(\Omega)}^2 = -(\mathbb{S}_\l(\hpsimn),\nabla \umn)_\Omega + (f,\umn)_{\Omega},$$
         where we have used that the convective term vanishes thanks to $\div\,   \umn=0$.
         Then, using Young's inequality, we have after integration over $(\I)$ and applying the supremum over all $t \in (\I)$ that
         \begin{equation}\label{Eq:NindepU} \sup_{t\in(\I)} \|\umn\|_{L^2(\Omega)}^2 +  \|\nabla \umn\|_{{L^2(\Omega_T)}}^2 \leq C(M,\l,u_0,f).
         \end{equation}
        Using the definition of $\umn$ via its expansion in terms of the divergence-free Galerkin basis functions for the velocity field, we further obtain that
        $$\sup_{t,i} \Big( |c_i^{m,n}(t)|+\Big|\ddt c_i^{m,n}(t)\Big|\Big) \leq C(m,\l,u_0,f,M).$$
        
         Similarly, multiplying the $i$-th equation in \eqref{Eq:Galerkin}$_2$ by $d_i^{m,n}(t)$ and summing with respect to $i=1,\dots,n$, we deduce that
         \begin{equation}\label{Eq:TestedFP}\begin{aligned} &(M^m\pt( k*[\hpsimn-\hpsimn_0]),\hpsimn)_\O + \|\nabla \hpsimn\|_{L^2_{M^m}(\O)}^2 + (M^m,|\nablaq \hpsimn|^2)_\O  \\ &\quad = ({M}\Lambda_\l(\hpsimn)(\nabla \umn) q,\nablaq \hpsimn)_\O,
         \end{aligned}\end{equation}
         where we have used again that $\div\,   \umn=0$ to eliminate the convective term. 
         The first term in \cref{Eq:TestedFP} can be estimated using Alikhanov's inequality, see \cref{Eq:Chain}, as follows:
         $$\big(M^m\pt( k*[\hpsimn-\hpsimn_0]),\hpsimn\big)_\O \geq\frac12 \pt \Big(k*\Big[ \| \hpsimn\|_{L^2_{M^m}(\O)}^2-\| \hpsimn_0\|_{L^2_{M^m}(\O)}^2\Big]\Big).$$
         Using H\"older's inequality and Young's inequality,  the right-hand side of \cref{Eq:TestedFP} can be bounded by  
         \begin{align*}
            ({M}\Lambda_\l(\hpsimn)(\nabla \umn) q,\nablaq \hpsimn)_\O &\leq \frac12 \|\nablaq \hpsimn\|^2_{L^2_{M^m}(\O)} \\
            &\quad + C(\l) \|\nabla \umn\|_{L^\infty(\Omega)}^2  \|\hpsimn\|_{L^2_{M^m}(\O)}^2,
         \end{align*}
         where we have used that $M/M^m$ is nonnegative and bounded from above by 1 to obtain the $M^m$-weighted norms on the right-hand side.
         Using the smoothness of the basis functions and noting the Sobolev embedding $W^{d+1,2}(\Omega) \hookrightarrow W^{1,\infty}(\Omega)$ and the bound \eqref{Eq:NindepU} on the functions $\umn$, it follows by norm-equivalence in finite-dimensional spaces that
         \begin{equation}\label{Eq:BoundNablaUmn}\|\nabla \umn\|_{L^\infty(0,T;L^\infty(\Omega))} \leq C(m,\l,T).
         \end{equation}
         Therefore, we deduce that
                  $$\begin{aligned}  \ddt \Big(k*\Big[\| \hpsimn\|_{L^2_{M^m}(\O)}^2-\| \hpsimn_0\|_{L^2_{M^m}(\O)}^2\Big]\Big) +  \|\nablaxq \hpsimn\|_{L^2_{M^m}(\O)}^2 \\
                  \leq  C(m,\l)  \|\hpsimn\|_{L^2_{M^m}(\O)}^2.
         \end{aligned}$$
         At this point, we convolve this inequality with the resolvent kernel $\tilde k$ and apply the inverse convolution property \cref{FracProp} to deduce for the leading term on the left-hand side of the inequality that
         $$\tilde k*\pt\Big(k*\Big[ \|\hpsimn\|_{L^2_{M^m}(\O)}^2-\|\hpsimn_0\|_{L^2_{M^m}(\O)}^2\Big]\Big) =  \|\hpsimn\|_{L^2_{M^m}(\O)}^2- \|\hpsimn_0\|_{L^2_{M^m}(\O)}^2.$$
         Together with the other terms in the inequality, we obtain         \begin{equation}\label{Eq:PreEstMN}
         \begin{aligned}   
         \|\hpsimn(t)\|_{L^2_{M^m}(\O)}^2- \|\hpsimn_0\|_{L^2_{M^m}(\O)}^2 + \Big(\tilde k*\|\nablaxq \hpsimn\|_{L^2_{M^m}(\O)}^2\Big)(t) \\ 
         \leq C(m,\l)  \Big(\tilde k*\|\hpsimn\|_{L^2_{M^m}(\O)}^2\Big)(t).
         \end{aligned}\end{equation}
Using the fact that $\tilde k$ is nonincreasing, the second term on the left can be bounded by
\begin{align}\label{Eq:EstimateGradKernel} 
\begin{aligned}
\big(\tilde k*\|\nablaxq \hpsimn\|_{L^2_{M^m}(\O)}^2\big)(t) &\geq \tilde k(t)\|\nablaxq \hpsimn\|_{L^2(0,t;L^2_{M^m}(\O))}^2\\ 
&\geq \tilde k(T)\|\nablaxq \hpsimn\|_{L^2(0,t;L^2_{M^m}(\O))}^2.
\end{aligned}
\end{align}
Therefore, we deduce that
       $$\begin{aligned} &  \|\hpsimn(t)\|_{L^2_{M^m}(\O)}^2 +\tilde k(T)\|\nablaxq \hpsimn\|_{L^2(0,t;L^2_{M^m}(\O))}^2 \\ & \quad \leq  \|\hpsim_0\|_{L^2_{M^m}(\O)}^2+C(m,\l)  \Big(\tilde k*\|\hpsimn\|_{L^2_{M^m}(\O)}^2\Big)(t),
         \end{aligned}$$
         where we have also estimated $\|\hpsimn_0\|_{L^2_{M^m}(\O)}^2$ by $\|\hpsim_0\|_{L^2_{M^m}(\O)}^2$ using the initial condition \cref{Eq:DefinitionInitials2}.
        Then, we apply the Henry--Gronwall inequality \cite[Corollary 1]{fritz2022time} to absorb the second term on the right-hand side into the left-hand side of the inequality. Thus, we obtain the $n$-uniform estimate
\begin{equation*}\begin{aligned} &  \|\hpsimn(t)\|_{L^2_{M^m}(\O)}^2 +  \|\nablaxq \hpsimn\|_{L^2(\I;L^2_{M^m}(\O))}^2  \leq C(m,\l,T) \|\hpsim_0\|_{L^2_{M^m}(\O)}^2.
         \end{aligned}
         \end{equation*}
        At this point, we use that $M^m \geq 1/m$; see \cref{Def:ApproxMax} for the definition of the approximated Maxwellian $M^m$. That completes the proof of the lemma. 
\end{proof}

The $n$-uniform estimates yield weakly converging subsequences. As we want to pass to the limit $n \to \infty$ in the Galerkin system \cref{Eq:Galerkin} which involves nonlinearities, we require a strong convergence result, and we shall apply the compact embedding \cref{thm:nonlocal_aubin_measures} to prove this. To this end, we begin by bounding the nonlocal time derivative of $\hpsimn$.

\begin{Lemma} \label{Lem:NindepTime} Let Assumption \ref{Assumptions} hold. For the nonlocal time derivative of the Galerkin solution $\hpsim$ of \cref{Eq:Galerkin}, the following estimate holds:
\begin{equation} \label{Eq:BoundDerivative} \|M^m\pt (k*[\hpsimn-\hpsimn_0])\|_{L^2(\I;[W^{1,2}(\O)]^*)} \leq C(m,\l,T).
\end{equation}
     \end{Lemma}
     \begin{proof}
        We consider an arbitrary element $\zeta \in L^2(\I;W^{1,2}(\O))$ and test the Galerkin equation of $\hpsimn$, see \cref{Eq:Galerkin}$_2$, with the orthogonal projection of $\zeta$ onto the space that is spanned by the first $m$ eigenfunctions $\{\zeta_i^m\}$, that is, $\widehat\Pi_n \zeta:=\sum_{i=1}^n (T_\l(\zeta),\zeta_i^m)_\O \zeta_i^m.$ We apply H\"{o}lder's inequality to the inner products to deduce that
$$\begin{aligned}
    &\langle M^m\pt( k*[\hpsimn-\hpsimn_0]),\Pi_n \zeta \rangle \\ & = - (M^m\nabla\hpsimn,\nabla\Pi_n \zeta)_{\O_T}+ (M^m\umn\hpsimn,\nabla\Pi_n \zeta)_{\O_T} -(M^m\nablaq \hpsimn,\nablaq \Pi_n \zeta)_{\O_T} \\&\quad +({M} \Lambda_\l(\hpsimn)(\nabla \umn) q,\nablaq \Pi_n \zeta)_{\O_T} \\
    &\leq \|M^m\hpsimn\|_{L^2(\I;W^{1,2}(\O))} \|\zeta\|_{L^2(\I;W^{1,2}(\O))}\\
    &\quad +  \|\umn\|_{L^\infty(\Omega_T)} \|M^m\hpsimn\|_{L^2(\O_T)} \|\zeta\|_{L^2(\I;W^{1,2}(\O))} \\
    &\quad +C(\l)  \|q\|_{L^\infty(\D)} \|\nabla \umn\|_{L^\infty(\Omega_T)} \|\hpsimn\|_{L^2(\I;L^2_{M^m}(\O))}\|\zeta\|_{L^2(\I;W^{1,2}(\O))}.
\end{aligned}$$
Then, using the $n$-independent estimates of \cref{Lem:Nindep} and the Lipschitz continuity of $M^m$, we obtain
$$\begin{aligned}
    \langle M^m\pt( k*[\hpsimn-\hpsimn_0]),\Pi_n \zeta \rangle&\leq C(m,\l,T) \|\zeta\|_{L^2(\I;W^{1,2}(\O))}.
\end{aligned}$$
    Thus, taking the supremum over all functions $\zeta \in L^2(\I;W^{1,2}(\O))$, we arrive at the desired bound on $M^m\pt( k*[\hpsimn-\hpsimn_0])$.
     \end{proof}

     Having proved the $n$-independent a priori estimates stated in \cref{Lem:Nindep} and \cref{Lem:NindepTime}, we are now able to pass to the limit $n\to\infty$ in the Galerkin system \cref{Eq:Galerkin}, \cref{Eq:DefinitionInitials} in the next lemma.

 \begin{Lemma}
     Let Assumption \ref{Assumptions} hold. Then, as $n \to \infty$, the limit functions $(u^m,\hpsi^m)$ satisfy the system 
     \begin{align} \label{Eq:GalerkinM}\begin{aligned}
        &\langle \pt\um,w_i \rangle  -(\Gamma_\l(|\um|^2)\um \otimes \um,\nabla w_i)_\Omega \\
        &\quad + (\nabla \um,\nabla w_i)_\Omega +(\mathbb{S}_\l(\hpsim),\nabla w_i)_\Omega =( f,w_i)_\Omega, \\
        & \langle M^m\pt(k*[\hpsim-\hpsim_0]),\zeta \rangle - (M^m\um\hpsim,\nabla\zeta)_{\O}  + (\nabla(M^m\hpsim),\nabla\zeta)_{\O} \\ &\quad +(M^m\nablaq \hpsim,\nablaq \zeta)_{\O}=({M}\Lambda_\l(\hpsim)(\nabla \um) q,\nablaq \zeta)_{\O},
    \end{aligned}\end{align}
    for any $i=1,\dots,n$ and $\zeta \in W^{1,2}(\O)$, together with the initial condition $\um(0)=u_0^m$ for the velocity and $(k*[\hpsim-\hpsim_0])(0)=0$ for the probability density.
 \end{Lemma}    
 \begin{proof}
     Using the $n$-independent bounds, see \Cref{Lem:Nindep} and \cref{Lem:NindepTime}, we obtain the weak/weak-$*$ convergences
     $$\begin{aligned}
     c_i^{m,n} &\overset{*}{\rightharpoonup} c_i^m &&\text{weakly-$*$ in } W^{1,\infty}(\I), \\
         \hpsimn &\rightharpoonup \hpsim &&\text{weakly in } L^2(\I;W^{1,2}(\O)), \\
         \hpsimn &\overset{*}{\rightharpoonup} \hpsim &&\text{weakly-$*$ in } L^\infty(\I;L^2(\O)), \\
         M^m\pt (k*[\hpsimn-\hpsimn_0]) &\rightharpoonup M^m\pt (k*[\hpsim-\hpsim_0]) &&\text{weakly in } L^2(\I;[W^{1,2}(\O)]^*), \\
         \mathbb{S}_\l(\hpsimn) &\overset{*}{\rightharpoonup} \mathbb{S}_\l(\hpsim) &&\text{weakly-$*$ in } L^\infty(\I;L^\infty(\Omega)^{d\times d}),
     \end{aligned}$$
     where we note that the limit of the nonlocal derivative is indeed of this form because of linearity of the convolution operator and using integration by parts, see \cite[Proposition 3.5]{li2018some}.
     Then, we use the Aubin--Lions compactness lemma (cf. \cref{Eq:aubinclassic}) and its nonlocal counterpart \cref{thm:nonlocal_aubin_measures} to infer the strong convergences       
     $$\begin{aligned}
     c_i^{m,n} &\to c_i^m &&\text{strongly in } C([\I]),  \\
         \umn &\to \um &&\text{strongly in } C([\I];W^{1,2}_{0,\div  }(\Omega) \cap W^{d+1,2}(\Omega)^d ), \\
         \hpsimn &\to \hpsim &&\text{strongly in } L^2(\O_T).
     \end{aligned}$$
     Taking the limit $n \to \infty$ in the Galerkin equations \cref{Eq:Galerkin} then gives the system \cref{Eq:GalerkinM} stated in the lemma. Regarding the attainment of the initial conditions, we argue as follows. By the strong convergence of $\umn$ stated above, we have $\umn(0) \to \um(0)$ in $W^{1,2}_{0,\div  }(\Omega) \cap W^{d+1,2}(\Omega)^d$, from which it follows that $\umn(0)=\um_0$ and thus $\um(0)=\um_0$ by the uniqueness of the limits. Moreover, by the previous convergences, we have $$\begin{aligned}
         M^m k*[\hpsimn-\hpsimn_0] \overset{*}{\rightharpoonup} M^m k*[\hpsim-\hpsim_0] ~ \text{ weakly-$*$ in } &
         L^\infty(\I;L^2(\O)) \\ &\cap H^1(\I;[W^{1,2}(\O)]^*).
     \end{aligned}$$
    As the embedding operator from $W^{1,2}(\O)$ into $L^2(\O)$ is compact, its transpose, mapping $[L^2(\O)]^*$ (which is isometrically isomorphic to $L^2(\O)$) into $[W^{1,2}(\O)]^*$ is a compact linear operator. Hence, applying the compact embedding \cref{Eq:aubinclassic}$_2$ with $X=L^2(\O)$, $Z=[W^{1,2}(\O)]^*$ and $Y=Z$, it follows that
$$\begin{aligned}
         M^m k*[\hpsimn-\hpsimn_0] \to M^m k*[\hpsim-\hpsim_0] ~ \text{ strongly in } &C([\I];[W^{1,2}(\O)]^*),
     \end{aligned}$$
     which gives, inserting $t=0$, that
     $M^m (k*[\hpsim-\hpsim_0])(0)=0$ in $[W^{1,2}(\O)]^*$. 
 \end{proof}

\subsection{Step 3: Uniform a priori estimates in $m$ and the limit process $m \to \infty$}

In Step 2 we passed to the limit $n \to \infty$ in the $(\l,n,m)$-th Galerkin system \cref{Eq:Galerkin}, which resulted in the $(\l,m)$-th approximation \cref{Eq:GalerkinM}. Next, we derive $m$-uniform bounds and take the limit $m \to \infty$ in \cref{Eq:GalerkinM}. A key step in the argument is to prove the nonnegativity of $\hpsim$. While a probability density function is by definition nonnegative, it does not automatically follow that $\hpsim$ satisfies this property; instead, the nonnegativity of $\hpsim$ has to be inferred from the nonnegativity of the initial datum $\hpsim_0$.
Having shown the nonnegativity of $\hpsim$, we shall be able to test with the logarithm of $\hpsim+\delta$ for some small $\delta >0$.

\begin{Lemma}\label{lem:nonn}
   Let Assumption \ref{Assumptions} hold. Then $\hpsim$ is nonnegative.
\end{Lemma}
\begin{proof}
The nonnegative part of a function $w$ is defined by $[w]_+:=\max\{0,w\}$. As $[\cdot]_{+}$ is a Lipschitz-continuous function over $\mathbb{R}$, we deduce that $-[-\hpsim]_+$ preserves the spatial $W^{1,2}$-regularity and the integrability-in-time of $\hpsim$. Therefore, $-[-\hpsim]_+$ is a valid test function in \cref{Eq:GalerkinM}$_2$.
 Indeed, testing \cref{Eq:GalerkinM}$_2$ with $-[-\hpsim]_+$ gives
\begin{equation}\label{Eq:GalerkinTestedNegativePart}\begin{aligned}
    &\langle M^m \pt(k*[-\hpsim+\hpsim_0]),[-\hpsim]_+ \rangle - (M^m\um(-\hpsim),\nabla[-\hpsim]_+)_{\O}   \\
    &\quad + (M^m\nabla(-\hpsim),\nabla[-\hpsim]_+)_{\O}+(M^m\nablaq (-\hpsim),\nablaq [-\hpsim]_+)_{\O} \\&=(M^m\Lambda_\l(-\hpsim)(\nabla \um) q,\nablaq [-\hpsim]_+)_{\O}.
\end{aligned}\end{equation}
Regarding the first term of \cref{Eq:GalerkinTestedNegativePart}, we note that the function $H(y)=\frac12 (y_+)^2$ is convex with derivative $H'(y)=y_+$. Thus, we infer from Alikhanov's inequality \cref{Eq:Chain} that
$$\begin{aligned} 
&\langle M^m \pt (k* [-\hpsim+\hpsim_0]),[-\hpsim]_+ \rangle \\
&\quad \geq \frac12 \pt\Big(k*\Big[ \|[-\hpsim]_+\|^2_{L^2_{M^m}(\O)}-\|[-\hpsim_0]_+\|^2_{L^2_{M^m}(\O)}\Big]\Big) \\
&\quad = \frac12 \pt(k* \|[-\hpsim]_+\|^2_{L^2_{M^m}(\O)}),
\end{aligned}$$
where we have used that $[-\hpsim_0]_+=0$; this equality follows from the definition of $\hpsi_0^m$ (recall that 
$\hpsi_0^m:=\hpsi_0 M/M^m$), the assumed nonnegativity of $\hpsi_0$ (see Assumption \ref{Assumptions}), the nonnegativity of $M$ and the positivity of $M^m$.
The remaining terms on the left-hand side of \cref{Eq:GalerkinTestedNegativePart} are dealt with using a standard argument, resulting in 
$$\begin{aligned} 
(M^m\nabla(-\hpsim),\nabla [-\hpsim]_+)_\O &=\|\nabla[-\hpsim]_+\|^2_{L^2_{M^m}(\O)}, \\
((\um \cdot \nabla) (-\hpsim),[-\hpsim]_+)_\O &=0.\end{aligned}$$
It remains to bound the term on the right-hand side of \cref{Eq:GalerkinTestedNegativePart}. To this end, we use that $\|\nabla u^m\|_{L^\infty(\Omega_T)}\leq C(m,\l)$ in conjunction with the definition of $\Lambda_\l$, see \cref{Def:Tl}, to deduce that
\begin{align*} 
&(M^m\Lambda_\l(-\hpsim)(\nabla \um) q,\nablaq [-\hpsim]_+)_{\O} \\
&\quad \leq C(m,\l) \|[-\hpsim]_+\|_{L^2_{M^m}(\O)}^2 +  \frac{1}{2} \|\nablaq [-\hpsim]_+\|_{L^2_{M^m}(\O)}^2.
\end{align*}
Putting everything together, we arrive at the inequality
$$\begin{aligned}&\frac12 \pt(k* \|[-\hpsim]_+\|^2_{L^2_{M^m}(\O)})+\frac12 \|\nablaxq[-\hpsim]_+\|_{L^2_{M^m}(\O)}^2\leq C(m,\l)\|[-\hpsim]_+\|_{L^2_{M^m}(\O)}^2.
\end{aligned}$$
We convolve this inequality with $\tilde k$ and use the inverse convolution property \cref{FracProp}. Lastly, we apply the Henry--Gronwall inequality \cite[Corollary 1]{fritz2022time} to infer that
$\|[-\hpsim]_+\|^2_{L^2_{M^m}(\O)}=0$, which implies
$\hpsim \geq 0$ almost everywhere in $\O_T$.
\end{proof}

Our first $m$-independent estimate is for 
\[\varrho^m:= \int_D M^m \hpsim \dq\]
over the space-time domain $\Omega_T$. Note that, thanks to Lemma \ref{lem:nonn}, $\varrho^m \geq 0$ on $\Omega_T$.

\begin{Lemma} \label{Lem:Rhom}
    Let Assumption \ref{Assumptions} hold. Then $\rhom$ satisfies the uniform bound
    \begin{equation} \label{Eq:BoundsRhom}
    \|\rhom\|_{L^\infty(\Omega_T)}
 \leq C \|\rho_0\|_{L^\infty(\Omega)}.
\end{equation}
\end{Lemma}
\begin{proof}
By choosing a $q$-independent test function $\bar\zeta(x,t) \in W^{1,2}(\Omega)$ in the $m$-th approximation \cref{Eq:GalerkinM}$_2$, we have  the following nonlocal convection-diffusion equation that governs the evolution of $\rhom=(M^m ,\hpsim )_D$:
$$\langle \pt(k*[\rhom-\rhom_0]),\bar\zeta\rangle - (\um\rhom,\nabla \bar\zeta)_\Omega + (\nabla \rhom, \nabla\bar\zeta)_\Omega = 0 \quad \forall \bar\zeta \in W^{1,2}(\Omega),$$
supplemented with the initial condition $(k*(\rhom-\rhom_0))(0)=0$ where we define $\rhom_0$ by
$$0 \leq \rhom_0(x) := (M^m,T_\l(\psim_0(x,\cdot)))_D \leq (M, \psi_0(x,\cdot))_D=\rho_0(x).$$
Since $\div   \um =0$ and $\rho_0 \in L^\infty(\Omega)$ thanks to Assumption \ref{Assumptions}, we deduce the desired uniform bound \cref{Eq:BoundsRhom} by the maximum principle for nonlocal-in-time differential equations, see \cite[Theorem 3.2]{zacher2008boundedness}.
\end{proof}

Next, we derive a crucial $m$-uniform bound on the tuple $(\um,\hpsim)$ in suitable norms and then pass to the limit $m \to \infty$ in the $m$-th approximation \cref{Eq:GalerkinM} to the $\l$-truncated system \cref{Def:Truncation_Velocity}, \cref{Def:Truncation_PDF}. In contrast with the previous energy estimate, we derive at this step an entropy estimate by choosing an appropriate convex test function. This is essential for the final energy (in)equality and the handling of the nonlinearities present in the system. We stress that the key technical difference between the corotational model considered in \cite{fritz2024analysis} and the general noncorotational model studied here is that in the case of the noncorotational model the 
absence of an $m$ independent bound on $\|\nabla\um\|_{L^\infty(\Omega_T)}$ poses significant technical difficulties; in contrast, in the case of the corotional model, where the gradient of the velocity field in the Fokker--Planck equation is replaced by its skew-symmetric part, such a bound is not required in the proof of the existence of a global weak solution, and that then greatly simplifies the analysis.

\begin{Lemma}[$m$-uniform estimates]
    \label{Lem:Mindep} Let Assumption \ref{Assumptions} hold. For the tuple $(\um,\hpsim)$, the following a priori estimate holds:
     $$\begin{aligned} 
     & \|\um\|_{L^\infty(\I;L^2(\Omega))}^2 + \|\nabla \um\|_{L^2(\Omega_T)}^2+ \|\pt \um\|_{L^2(\I;[W^{1,2}_{0,\div  }(\Omega)]^*)}^2  + \|\um\|_{L^{2(d+2)/d}(\Omega_T)}^2 \\ &\quad +\|\mathbb{S}_\l(\hpsim)\|_{L^\infty(\I;L^\infty(\Omega))} +\|\hpsim\ln(\hpsim)\|_{L^{1}(\I;L^1_{M^m}(\O))}\\
     &\quad + \Big\|\nablaxq \sqrt{\hpsim}\Big\|_{L^2(0,T;L^2_{M^m}(\O))}^2\leq C(\l).
         \end{aligned}$$
\end{Lemma}
\begin{proof}
First, we test the approximated Navier--Stokes equation \cref{Eq:GalerkinM}$_1$ with $\um$, which gives
\begin{equation*}  
\begin{aligned}\frac12 \ddt \|u^m\|_{L^2(\Omega)}^2 +  \|\nabla u^m\|_{L^2(\Omega)}^2 &= (f,\um)_{\Omega} -(\mathbb{S}_\l(\hpsim),\nabla \um)_\Omega.  
\end{aligned}\end{equation*}
We shall bound the first term on the right-hand side; however, following the ideas in \cite{barrett2011existence} and \cite{bulicek2013existence}, we leave the second term on the right-hand side intact because its  current form plays a crucial role in cancelling the noncorotational drag term in the Fokker--Planck equation. We integrate the inequality over $(0,t)$ to deduce that
\begin{equation} \label{Eq:NSTested} 
\begin{aligned} &\frac12  \|u^m(t)\|_{L^2(\Omega)}^2 +  \frac12 \|\nabla u^m\|_{L^2(\Omega)}^2 \\ &\quad \leq \frac12  \|u_0^m(t)\|_{L^2(\Omega)}^2 +  \frac12\|f\|_{L^{2}(0,T;W^{-1,2}(\Omega)^d)}^{2} -\int_0^t (\mathbb{S}_\l(\hpsim),\nabla \um)_\Omega \ds.  
\end{aligned}\end{equation}
Before turning our attention to the $m$-th approximation \cref{Eq:GalerkinM}$_2$ of the nonlocal Fokker--Planck equation, we first define the convex function $G_\delta \in C^\infty((-\delta,\infty)) \cap C^0([-\delta,\infty))$ by $$G_\delta(s):=(s+\delta)\ln (s+\delta)+\tfrac{1}{e}\geq 0,$$ where $\delta>0$ is arbitrary and we set $G(\hpsim):=G_0(\hpsim)$. Further, we define the function $T_{\delta,\l}(s):=\int_0^s \frac{\Lambda_\l(t)}{t+\delta} \dt$ that converges to $T_\l$ in $C([0,\infty))$ as $\delta \to 0$.
We test the approximated Fokker--Planck equation \cref{Eq:GalerkinM}$_2$ with $G_\delta'(\hpsim)=\ln(\hpsim+\delta)+1$, which gives
\begin{equation}\label{Eq:LogTestedFP} \begin{aligned} &(M^m \pt (k*[\hpsim-\hpsim_0]),G_\delta'(\hpsim))_\O  +(M^m\nabla\hpsim,\nabla G_\delta'(\hpsim))_\O \\
&\qquad + ( M^m \nablaq \hpsim,\nablaq G_\delta'(\hpsim))_\O \\ &\quad =(M^m \um \hpsim,\nabla G_\delta'(\hpsim))_\O+({M} \Lambda_\l(\hpsim)(\nabla \um)q,\nablaq G_\delta'(\hpsim))_\O.
\end{aligned}\end{equation}
Regarding the first term of the left-hand side of \cref{Eq:LogTestedFP}, we apply Alikhanov's inequality for convex functions, see \cref{Eq:Chain}, to estimate it from below as follows:
$$\begin{aligned} (M^m \pt (k*[\hpsim-\hpsim_0]),G_\delta'(\hpsim))_\O  \geq \pt\left(k* \int_\O M^m (G_\delta(\hpsim)-G_\delta(\hpsim_0)) \dxq \right)\\ =\pt\Big(k*\Big[ \|G_\delta(\hpsim)\|_{L^1_{M^m}(\O)}-\|G_\delta(\hpsim_0)\|_{L^1_{M^m}(\O)}\Big]\Big) .
\end{aligned}$$
For the second term on the left-hand side of \cref{Eq:LogTestedFP}, we have
$$\lim_{\delta\to 0}(M^m\nabla \hpsim,\nabla G_\delta'(\hpsim))_{\O}=\lim_{\delta\to 0}(M^mG_\delta''(\hpsim),|\nabla \hpsim|^2)_{\O}=4 \Big\|\nabla \sqrt{\hpsim}\Big\|_{L^2_{M^m}(\O)}^2,$$
and in the same manner we handle the third term on the left-hand side of \cref{Eq:LogTestedFP}. The first term on the right-hand side of \cref{Eq:LogTestedFP} is zero after taking the limit $\delta \to 0$, because $$\nabla G_\delta'(\hpsim)=G_\delta''(\hpsim) \nabla\hpsim=\frac{\nabla\hpsim}{\hpsim+\delta},$$  and thus, by the monotone convergence theorem, we deduce that
$$-\lim_{\delta\to 0}(u^m,M^m\hpsim\nabla G_\delta'(\hpsim))_{\O}=-(u^m,M^m\nabla \hpsim)_\O=(\div\,   u^m,M^m \hpsim)_\O=0.$$
Furthermore, for the remaining term on the right-hand side of \cref{Eq:LogTestedFP}, we compute the limit $\delta \to 0$ as follows:
\begin{equation*}\begin{aligned}  \lim_{\delta\to 0}({M} \Lambda_\l(\hpsim) (\nabla u^m)q , \nablaq G_\delta'(\hpsim))_{\O} &=\lim_{\delta\to 0} ({M} q\nabla u^m, \nablaq T_{\delta,\l}(\hpsim))_{\O} \\ &=-\sum_{j=1}^K(\nablaqj({M}\otimes  q^j)\nabla u^m, T_\l(\hpsim))_{\O} \\
&=(\mathbb{S}_\l(\hpsim),\nabla \um)_\O,
\end{aligned}\end{equation*}
where we used integration by parts in the second step and the boundary term vanishes thanks to fact that ${M}|_{\partial D}=0$. 
Taking the limit $\delta \to 0$ in \cref{Eq:LogTestedFP} and using the above considerations, we obtain the upper bound
\begin{equation*} \begin{aligned} &\pt\Big(k*\Big[ \|G(\hpsim)\|_{L^1_{M^m}(\O)}-\|G(\hpsim_0)\|_{L^1_{M^m}(\O)}\Big]\Big)  + 4 \Big\|\nablaxq \sqrt{\hpsim}\Big\|_{L^2_{M^m}(\O)}^2\\ &\quad\leq (\mathbb{S}_\l(\hpsim),\nabla \um)_\O.
\end{aligned}\end{equation*}
At a first glance, the next natural step might seem to be to convolve this estimate with $\tilde k$ and use the inverse convolution property \cref{FracProp} to obtain that $G(\psim)$ is bounded in $L^\infty(0,T;L^1_{M^m}(\O))$. However, we have to take care of the term on the right-hand side of the inequality. It is essential that it cancels with the  right-hand side of \cref{Eq:NSTested}  as otherwise the term cannot be controlled in terms of the available bounds. Therefore, we do not convolve but integrate. As a result, we only obtain that $G(\psim)$ is bounded in $L^1_{M^m}(\O_T)$ but not necessarily in $L^\infty(0,T;L^1_{M^m}(\O))$; i.e., 
\begin{equation}\label{Eq:LogTestedFP2} \begin{aligned} & \Big(k*\|G(\hpsim)\|_{L^1_{M^m}(\O)}\Big)(t) + 4\Big\|\nablaxq \sqrt{\hpsim}\Big\|_{L^2(0,t;L^2_{M^m}(\O))}^2\\[-.2cm] &\quad \leq \|k\|_{L^1(0,T)}\|G(\hpsim_0)\|_{L^1_{M^m}(\O)} + \int_0^t (\mathbb{S}_\l(\hpsim),\nabla \um)_\O \ds,
\end{aligned}\end{equation}
where we further used that $(k*1)(t)=\|k\|_{L^1(0,t)} \leq \|k\|_{L^1(0,T)}$. Since $k$ is nonincreasing, its infimum is given by $k(T)$ and thus, we can estimate the first term on the left-hand side of \cref{Eq:LogTestedFP2} from below by
$$\Big(k*\|G(\hpsim)\|_{L^1_{M^m}(\O)}\Big)(t) \geq k(T) \|G(\hpsim)\|_{L^1(0,t;L^1_{M^m}(\O))}.$$
We add the estimate \cref{Eq:LogTestedFP2} to the integrated inequality \cref{Eq:NSTested} that we obtained from testing the Navier--Stokes equations. We observe that the nonlinear terms involving the $\l$-truncated extra-stress tensor on the right-hand sides of the two inequalities cancel each other. More precisely, we obtain the $m$-uniform bound 
\begin{equation}\begin{aligned}
    &\frac12 \|u^m(t)\|_{L^2(\Omega)}^2 + \frac12  \|\nabla u^m\|_{L^2(\Omega_T)}^2+ \|G(\hpsim)\|_{L^1_{M^m}(\O_T)} +4 \Big\|\nablaxq \sqrt{\hpsim} \Big\|_{L^2_{M^m}(\O_T)}^2 \\ &\,\,\, \leq \frac12 \|u^m_0\|_{L^2(\Omega)}^2 + \|k\|_{L^1(0,T)}\|G(\hpsim_0)\|_{L^1_{M^m}(\O)}+\frac12\|f\|_{L^{2}(0,T;W^{-1,2}(\Omega)^d)}^{2} \leq C(\l),
\end{aligned}\label{Eq:Entropy} \end{equation}
where we additionally used that $\um_0$ is the $L^2_{0,\rm{div}}(\Omega)$ orthogonal projection of $u_0$ in the final inequality. We also recall that $\hpsim_0$ was defined as $\hpsim_0=\hpsi_0M/M^m$, see \cref{Eq:DefinitionInitials}, to remove the $m$-dependency on the right-hand side in the transition to the right-hand side of the final inequality. 
By standard $L^p$-$L^q$ interpolation inequalities, we conclude by
\begin{equation} \label{Eq:InterpolatedUm}
    \|\um\|_{L^{2(d+2)/d}(\Omega_T)} \leq C(\l).
\end{equation}
By the definition of the $\l$-truncated stress tensor $\mathbb{S}_\l$, see \cref{sl}, and the presence of the truncation operator, we can directly deduce
$$\|\mathbb{S}_\l(\hpsim)\|_{L^\infty(0,T;L^\infty(\Omega))} \leq C(\l).$$

At this point, only the estimate of $\pt \um$ as stated in the estimate of the lemma is missing. Indeed, we consider an arbitrary element $w \in L^2(0,T;W^{1,2}_{0,\div  }(\Omega))$. Thanks to the truncation in the convection term of $u^m$, see \cref{Eq:GalerkinM}$_1$, and the boundedness of $\mathbb{S}_\l(\hpsim)$, we obtain the following inequality 
$$\begin{aligned}
|\langle\pt u^m,\Pi_m w \rangle|
&= \big|(\Gamma_\l(|\um|^2) \um \otimes \um- \nabla \um- S_\ell(\hpsim),\nabla \Pi_m w)_{\Omega_T} + (f,\Pi_m w)_{\Omega_T} \big|
\\ &\leq C(\l) \|w\|_{L^2(\I;W^{1,2}_{0,\div  }(\Omega))},
\end{aligned}$$
where we have bound each of the terms appearing on the right-hand side by means of H\"older's inequality and the $m$-uniform energy estimate \cref{Eq:Entropy}.
By taking the supremum over all functions in $L^2(\I;W^{1,2}_{0,\div  }(\Omega))$, we obtain that 
$\pt u^m$ is bounded in $L^2(\I;[W^{1,2}_{0,\div  }(\Omega)]^*)$.
 \end{proof} 
 
Next, we consider the limit process $m \to \infty$ in the $m$-th approximation \cref{Eq:GalerkinM} to prove that the limit tuple $(\ul,\hpsil)$ of $(\ulm,\hpsilm)$ satisfies the $\l$-approximation \cref{Eq:Lapproximation}. Hence, we prove \cref{Lem:ExistenceLapprox} as already stated above. By the $m$-uniform bounds for $(\hpsim,\um)$ as stated in \cref{Lem:Mindep}, we can already extract a weakly/weakly-$*$ subsequence that. Together with the Aubin--Lions lemma to obtain strong convergence, this allows us to pass to the limit in the $m$-approximated Navier--Stokes equations. This is shown in the next lemma.
\begin{Lemma} \label{Lem:LimitNSm}
 Let Assumption \ref{Assumptions} hold. Then, the tuple $(\ul,\hpsil)$ solves the truncated Navier--Stokes equations \cref{Eq:GalerkinM}$_1$.
\end{Lemma}

\begin{proof}
We begin by reiterating our notational conventions that we are suppressing the index $\ell$ associated with the truncation process, so in the proof below we shall write $(\um,\hpsim)$ instead of $(u^{m,\ell}, \hpsi^{m,\ell})$ and $(u,\hpsi)$ instead of $(u^\ell,\hpsi^\ell)$.

By means of the $m$-uniform energy estimate in \cref{Lem:Mindep} and the classical Aubin--Lions lemma \cref{Eq:aubinclassic}, we deduce the existence of limit functions $u$ and $\hpsi$ with the convergences
    $$\begin{aligned}
        u^m &\overset{*}{\rightharpoonup} u &&\text{weakly-$*$ in } L^\infty(\I;L^2_{0,\div  }(\Omega)) \cap L^2(\I;W^{1,2}_{0,\div  }(\Omega)) \\
        & && \hspace{7cm} \cap L^{2(d+2)/d}(\Omega_T)^d, \\
        \pt u^m &\rightharpoonup \pt u &&\text{weakly in } L^{2}(\I;H_{0,\div  }^{-1}(\Omega)), \\
        u^m &\to u &&\text{strongly in } L^2(\I;L^r(\Omega)^d) \cap C([0,T];[W^{1,2}_{0,\div  }(\Omega)]^*), \\
        \mathbb{S}_\l(\hpsim) &\overset{*}{\rightharpoonup} \mathbb{S}_\l(\hpsi) &&\text{weakly-$*$ in } L^\infty(\I;L^\infty(\Omega)^{d\times d}),
    \end{aligned}$$
    where $r<2^*:=2d/(d-2)$ is less than the Sobolev conjugate $2^*$ of $2$, ensuring the compactness of the first of the two embeddings in the Gelfand triple $W^{1,2}_{0,\div  }(\Omega) \dhookrightarrow L^r(\Omega) \con [W^{1,2}_{0,\div  }(\Omega)]^*$ involved in the application of the Aubin--Lions lemma. 
Then it is standard to let $m \to \infty$ in the $m$-th Galerkin system of the Navier--Stokes equations \cref{Eq:GalerkinM}$_1$ to deduce that $u$ solves the truncated Navier--Stokes equations. Lastly,  the strong convergence $\um(0) \to u(0)$ in $[W^{1,2}_{0,\div  }(\Omega)]^*$ guarantees that $u$ satisfies the initial condition.
\end{proof}

Having passed the limit $m \to \infty$ in the approximated Navier--Stokes equations \cref{Eq:GalerkinM}$_1$ back in \cref{Lem:LimitNSm}, in order to prove that $(\ul,\hpsil)$ satisfies the $\l$-truncated system \cref{Def:Truncation_Velocity}, \cref{Def:Truncation_PDF} it remains to pass to the limit $m \to \infty$ in the $m$-th approximated nonlocal Fokker--Planck \cref{Eq:GalerkinM}$_2$ equation. To do this, we begin by noting that, thanks to \eqref{Eq:Entropy},
\begin{equation}\label{Eq:Entropy3}\|M^m G(\hpsim)\|_{L^1(\O_T)}\leq C(\l),
\end{equation}
where $G(\hpsim(t))=\hpsim\ln(\hpsim)+1/e$. In \cite{bulicek2013existence} this bound was used in conjunction with a compensated compactness argument based on the div-curl lemma, see \cite[Lemma A.1]{bulicek2013existence}, to infer strong convergence of $\hpsim$. In our setting, the presence of the nonlocal time derivative in the Fokker--Planck equation obstructs the application of the div-curl lemma. We are therefore forced to develop a different approach. We shall, instead, bound the nonlocal time derivative and apply the compactness result \cref{thm:nonlocal_aubin_measures} that we derived in \cref{Sec:Prelim}. Using this procedure, we obtain the following strong convergences for the sequences $\{\hpsim\}_{m \in \mathbb{N}}$ and $\{\psim\}_{m \in \mathbb{N}}$, where $\psim:=M^m\hpsim$, as $m \to \infty$.


\begin{Lemma} \label{Lem:StrongPsi} Let Assumption \ref{Assumptions} hold. Then, the following convergence results hold for any $s<\infty$:
    $$\begin{aligned}
\psim &\to \psi \quad &&\text {strongly in } L^s(\Omega_T ; L^{1}(\D)), \\
\mathbb{S}_\l(\hpsim) &\to \mathbb{S}_\l(\hpsi) \quad &&\text {strongly in } L^s(\Omega_T)^{d \times d}, \\
M^{m} \nablaxq\hpsim &\rightharpoonup M \nablaxq\hpsi \quad &&\text {weakly in } L^{1}(\O_T)^{d(K+1)}.
\end{aligned}$$
\end{Lemma}
\begin{proof}
The estimate \eqref{Eq:Entropy3} and the bound $M^m \leq C$ yield $\|G(\psim)\|_{L^1(\O_T)}\leq C(\l)$. Noting that the function $G(s) = s \ln(1 + s) +1/e$ is nonnegative, convex and has superlinear growth at infinity, De la Vall\'ee Poussin's criterion (see \cite[II.22]{dellacherie1978probabilities}) implies the uniform integrability of $(\psim)_m$ in $L^1(\O_T)$, that is,
\begin{equation}\label{Eq:UniformIntegrability}
\forall \varepsilon > 0 \quad \exists \delta > 0 \quad \forall m \in \mathbb{N} \quad \forall U \subset \O_T: |U| \leq \delta \implies \|\psim\|_{L^1(U)} \leq \varepsilon.
\end{equation}
Thus, by the Dunford--Pettis theorem, there exists a subsequence (not relabelled) and $\psi \in L^1(\O_T)$ such that
\[\psim \rightharpoonup \psi \quad \text{weakly in } L^1(\O_T).\]
Using the identity $\nabla_{x,q} \hpsim = 2\sqrt{\hpsim} \nabla_{x,q} \sqrt{\hpsim}$ and H\"older's inequality, we deduce
\begin{equation}\label{Eq:FirstBoundOfGrad}
\begin{aligned}
\|M^m \nabla_{x,q} \hpsim\|_{L^1(\O_T)} &= \Big\|2M^m \sqrt{\hpsim} \nabla_{x,q} \sqrt{\hpsim}\Big\|_{L^1(\O_T)} 
\\ &\leq 2\Big\|\sqrt{M^m}\sqrt{\hpsim}\Big\|_{L^2(\O_T)} \Big\|\sqrt{M^m}\nabla_{x,q} \sqrt{\hpsim}\Big\|_{L^2(\O_T)} 
\leq C(\l),
\end{aligned}
\end{equation}
where we used Lemma \ref{Lem:Rhom} and the bound \eqref{Eq:Entropy} in the last step. 
With this estimate, we are able to further bound the nonlocal derivative of $\hpsim$ and apply the compactness result \cref{thm:nonlocal_aubin_measures} to obtain a strong convergence result for $\hpsim$.  We start by considering an arbitrary element $\zeta \in C([\I];W^{1,\infty}(\O))$ as time-dependent test function. The reasoning behind this step is that we can multiply the $m$-approximated Fokker--Planck equation \cref{Eq:GalerkinM}$_2$ by a function $\eta \in C_c^1(0,T)$ and integrate the equation from $0$ to $T$; then, the density of $C_c^1(0,T) \otimes W^{1,2}(\O)$ in $L^2(0,T;W^{1,2}(\O)) (\supset C([\I];W^{1,\infty}(\O)))$ allows us to use $\zeta$ as a test function. We proceed as follows:
\[
\begin{aligned}
&\langle M^m\partial_t(k*[\hpsim-\hpsim_0]),\zeta \rangle \\
&\quad = (M^m\um\hpsim,\nabla \zeta)_{\O_T} - (M^m \nabla_{x,q}\hpsim,\nabla_{x,q}\zeta)_{\O_T} + (M\Lambda_\l(\hpsim)(\nabla \um) q,\nabla_q \zeta)_{\O_T} \\
&\quad \leq \|M^m\|_{L^\infty(\O)} \|\um\|_{L^\infty(0,T;L^2(\Omega))} \|\hpsim\|_{L^1(\O_T)} \|\nabla\zeta\|_{C([\I];L^\infty(\O))} \\
&\qquad + \|M^m \nabla_{x,q} \hpsim\|_{L^1(\O_T)} \|\nabla_{x,q} \zeta\|_{C([\I];L^\infty(\O))} \\
&\qquad + \|M\|_{L^\infty(\O)} \|\Lambda_\l(\hpsim)\|_{L^\infty(\O_T)} \|\nabla \um\|_{L^2(\Omega_T)} \|q\|_{L^\infty(D)} \|\nabla_q \zeta\|_{L^2(\O_T)} \\
&\quad \leq C(\l) \|\zeta\|_{C([\I];W^{1,\infty}(\O))}.
\end{aligned}
\]
Since  $M^m$ is time-independent, 
\[ M^m\partial_t(k*[\hpsim-\hpsim_0]) = \partial_t(k*[M^m\hpsim-M^m\hpsim_0]) = 
\partial_t(k*[\psim-\psim_0]),\]
and therefore, 
\[
\langle \partial_t(k*[\psim-\psim_0]),\zeta \rangle \leq C(\l) \|\zeta\|_{C([\I];W^{1,\infty}(\O))}.
\]
Hence, $\partial_t(k*[\psim-\psim_0])$ is bounded, uniformly in $m$, in $[C([\I]; W^{1,\infty}(\O))]^* = \mathcal{M}(0,T;[W^{1,\infty}(\Omega)]^*)$, 
the space of regular Borel measures on $[\I]$ with values in $[W^{1,\infty}(\O)]^*$ (cf. Thm. 2 in Ch. VI, Sec. 1 of Diestel and Uhl \cite{MR453964}).
Furthermore, $\psim$ is bounded in $L^1(0,T; W^{1,1}(\O))$ thanks to the equality
\[
\nabla_{x,q} \psim = \nabla_{x,q}(M^m\hpsim) = M^m \nabla_{x,q} \hpsim + \hpsim \nabla_{x,q} M^m,
\]
where the first term is bounded because of \cref{Eq:FirstBoundOfGrad}. As $M^m \in C^1(\overline{D})$ (since $M \in C^1(\overline D))$, see also \eqref{Def:ApproxMax}, and $\|\hpsim\|_{L^1(\O_T)} \leq C(\l)$, see \cref{Eq:Entropy} and the definition of $G$, the second term is also bounded uniformly in $m$.
We apply the compactness result \cref{thm:nonlocal_aubin_measures} with the Gelfand triple:
\[
X = W^{1,1}(\O) \quad \dhookrightarrow \quad Y = L^1(\O)\quad \hookrightarrow \quad Z = [W^{1,\infty}(\O)]^*.
\]
Therefore, $(\psim)_m$ is relatively compact in $L^1(0,T; L^1(\O))$, which implies the existence of a subsequence (not relabelled) such that:
\[
\psim \to \psi \quad \text{strongly in } L^1(\O_T).
\]

As, by hypothesis, $\psi_0 \geq 0$, and $\psi^m = M^m \hpsi^m$ and, according to Lemma \ref{lem:nonn}, $\hpsi^m \geq 0$, the inequality \eqref{Lem:Rhom} stated in Lemma \ref{Lem:Rhom} implies that
\[ \|\psi^m\|_{L^\infty(\Omega_T;L^1(D))} \leq C\|\psi_0\|_{L^\infty(\Omega;L^1(D))} .\]
Thus, 
$\psim$ is bounded in $L^\infty(\Omega_T; L^{1}(D))$ uniformly in $m$. This then, together with the strong convergence in $L^1(\O_T)$, implies that 
\[
\int_{\Omega_T} \|\psim(t) - \psi(t)\|_{L^1(D)}^s \dx \dt \leq \|\psim - \psi\|_{L^\infty(\Omega_T;L^1(D))}^{s-1} \|\psim - \psi\|_{L^1(\O_T)} \to 0.
\]
Thus, $\psim \to \psi$ strongly in $L^s(\Omega_T; L^1(D))$ as $m \to \infty$ for any $s \in [1,\infty)$.
Since $\hpsim = \psim/M^m$ and $M^m \to M$ uniformly, the strong convergence of $\psim$ implies the strong convergence
\begin{equation} \label{Eq:StrongHpsim} \hpsim \to \hpsi \quad \text{strongly in } L^s(\Omega_T;L^1(D))\end{equation}
for all $s \in [1,\infty)$. By the properties of the truncation operator $T_\l$ and the stress tensor $\mathbb{S}$, and using Lebesgue's dominated convergence theorem, we obtain:
\[
\mathbb{S}_\l(\hpsim) = \mathbb{S}(T_\l(\hpsim)) \to \mathbb{S}(T_\l(\hpsi)) = \mathbb{S}_\l(\hpsi) \quad \text{strongly in } L^1(\Omega_T)^{d \times d}.
\]
Interpolating between this strong convergence result and the $L^\infty(\Omega_T)$ bound on 
$\mathbb{S}_\l(\hpsim)$, it deduce the strong convergence of $\mathbb{S}_\l(\hpsim)$ in $L^s(\Omega_T)^{d \times d}$ to $\mathbb{S}_\l(\hpsi)$ as $m \to \infty$ for all $s \in [1,\infty)$. 

By the same reasoning as in \cref{Eq:FirstBoundOfGrad}, we obtain for any $U \subset \O_T$ with $|U|\leq \delta$ that
\begin{equation}\label{Eq:FirstBoundOfGrad2}
\begin{aligned}
\|M^m \nabla_{x,q} \hpsim\|_{L^1(U)} 
\leq 2 \|M^m\|_\infty \big\| \hpsim\big\|_{L^1(U)}^{1/2} \Big\|\sqrt{M^m}\nabla_{x,q} \sqrt{\hpsim}\Big\|_{L^2(U)}
\leq C(\l) \eps^{1/2},
\end{aligned}
\end{equation}
where $\eps$ is the same parameter as in \cref{Eq:UniformIntegrability}. Thus, $(M^m \nablaxq \hpsim)_m$ is uniformly integrable in $L^1(\O_T)^{d(K+1)}$, and thus we can extract a subsequence (not relabelled) such that
$$M^m \nablaxq \hpsim \rightharpoonup \phi \quad \text{weakly in } L^1(\O_T)^{d(K+1)}.$$
It remains to identify the weak limit $\phi$; we will show that $\phi = M \nabla_{x,q} \hpsi$. For any test function $\varphi \in C_c^\infty(\O_T)^{d(K+1)}$, we have by partial integration
\[
\begin{aligned}
\int_{\O_T} M^m \nabla_{x,q} \hpsim \cdot \varphi \dx \dq \dt &= - \!\int_{\O_T} \hpsim \, \nabla_{x,q} \cdot (M^m \varphi) \dx \dq \dt \\
&= -\! \int_{\O_T}\! \hpsim M^m \nabla_{x,q} \cdot \varphi \dx \dq \dt \\
& \quad \,-  \int_{\O_T}\! \hpsim \varphi \cdot \nabla_{x,q} M^m \dx \dq \dt.
\end{aligned}
\]
Now we pass to the limit $m \to \infty$. Since $M^m \nabla_{x,q} \hpsim \rightharpoonup \phi$ weakly in $L^1(\O_T)$, $\varphi$ is bounded, $M^m \to M$ uniformly, $\hpsim \to \hpsi$ strongly in $L^1(\O_T)$, we can easily pass to the limit on the left-hand side and in the first term on the right-hand side. It remains to deal with the second term on the right-hand side.
Since $\varphi$ has compact support in $\O_T$, its $q$-support is contained in some compact set $D' \subset D$. On $D'$, we have $\nabla_q M^m = \nabla_q M \in C(\overline{D'})$. Also, trivially, $\nabla_x M^m = \nabla_x M = 0$. Combined with the strong convergence $\hpsim \to \hpsi$ in $L^1(\O_T)$, we obtain:
\[
\int_{\O_T} \hpsim \varphi \cdot \nabla_{x,q} M^m \dx \dq \dt \to \int_{\O_T} \hpsi \varphi \cdot \nabla_{x,q} M \dx \dq \dt.
\]
Thus, in the limit $m \to \infty$ we have:
\[
\begin{aligned}
\int_{\O_T} \phi \cdot \varphi \dx \dq \dt &= - \int_{\O_T} \hpsi M \nabla_{x,q} \cdot \varphi \dx \dq \dt - \int_{\O_T} \hpsi \varphi \cdot \nabla_{x,q} M \dx \dq \dt \\
&=\int_{\O_T} M \nabla_{x,q} \hpsi \cdot \varphi \dx \dq \dt,
\end{aligned}
\]
for any  $\varphi \in C_c^\infty(\O_T)^{d(K+1)}$,
which implies that $\phi = M \nabla_{x,q} \hpsi$ almost everywhere on $\O_T$. Hence,
\[
M^m \nabla_{x,q} \hpsim \rightharpoonup M \nabla_{x,q} \hpsi \quad \text{weakly in } L^1(\O_T)^{d(K+1)}.
\]
\end{proof}

At this point, we have derived the necessary convergence properties of $\hpsim$ and $\um$ to pass to the limit the $m$-th approximated Fokker--Planck equation \cref{Eq:GalerkinM}$_2$, and thereby deduce the existence of a weak solution to the $\l$-truncated system \cref{Def:Truncation_Velocity}, \cref{Def:Truncation_PDF} as stated in \cref{Lem:ExistenceLapprox}.

\begin{proof}[Proof of \cref{Lem:ExistenceLapprox}]
We shall pass to the limit $m \to \infty$ in \eqref{Eq:GalerkinM}$_2$ to deduce that the nonlocal Fokker--Planck equation \cref{Def:Truncation_PDF} (with cutoff) is satisfied.
Passage to the limit $m \to \infty$ in the third and fourth term on the left-hand side of  \eqref{Eq:GalerkinM}$_2$ is straightforward, using the last weak convergence results stated in Lemma \ref{Lem:StrongPsi}. 
It remains to pass to the limit $m \to \infty$ in the first and second terms on the left-hand side of 
 \eqref{Eq:GalerkinM}$_2$ and the term on the right-hand side of  \eqref{Eq:GalerkinM}$_2$.
We note to this end that by interpolating between the weak convergence of $\um$ in $L^{2d/(d-2)}(\Omega_T)^d$ and its strong convergence in $L^2(0,T;L^r(\Omega)^d)$ for any $r \in [1,2d/(d-2))$ we have
$$\um \to u \quad\text{ strongly in } L^r(\Omega_T)^d$$
for all $r \in [1, 2d/(d-2))$. Combining this result and the strong convergence of $\psim$ in $L^s(\Omega_T;L^1(D))$ for any $s \in [1,\infty)$, see \cref{Lem:StrongPsi}, we deduce that
$$
\begin{aligned}
	M^{m} \hpsim \um=\psim \um &\to M\psi u && \text { strongly in } L^r(\Omega_T ; L^{1}(\D)^{d(K+1)}), \\
	\Lambda_{\l}(\hpsim) \nabla u^{m} &\rightharpoonup \Lambda_{\l}(\hpsi) \nabla u && \text { weakly in } L^2(\O_T)^{d \times d},
\end{aligned}
$$
where we used the strong convergence of $\hpsim$, see \cref{Eq:StrongHpsim}, and the Lebesgue dominated convergence theorem for the second convergence result.

Since all terms in \eqref{Eq:GalerkinM}$_2$ besides the first converge, we pass to the limit and obtain
$$\lim_{m \to \infty} \langle M^m\partial_t(k*[\hpsim-\hpsim_0]),\zeta \rangle = L(\zeta)$$ 
for all $\zeta \in C([\I];W^{1,\infty}(\O))$, where $L \in [C([\I];W^{1,\infty}(\O))]^*$ is to be identified.
To this end, consider the function $M k*[\hpsi-\hpsi_0]$. Since $\hpsim \to \hpsi$ strongly in $L^1(\O_T)$ and $k \in L^1(0,T)$, we have $M k*[\hpsim-\hpsim_0] \to M k*[\hpsi-\hpsi_0]$ strongly in $L^1(\O_T)$. For any $\eta \in C_c^1(0,T; W^{1,\infty}(\O))$, integration by parts yields
\[
\begin{aligned}
\langle \partial_t(M k*[\hpsi-\hpsi_0]), \eta \rangle = - \langle M k*[\hpsi-\hpsi_0], \partial_t \eta \rangle &= \lim_{m\to\infty} - \langle M^m k*[\hpsim-\hpsim_0], \partial_t \eta \rangle \\ &= \lim_{m\to\infty} \langle M^m\partial_t(k*[\hpsim-\hpsim_0]), \eta \rangle.
\end{aligned}
\]
As the right-hand side of this equality equals $L(\eta)$ and $C_c^1(0,T; W^{1,\infty}(\O))$ is dense in $C([\I];W^{1,\infty}(\O))$, it follows that $\langle \partial_t(M k*[\hpsi-\hpsi_0]), \zeta \rangle = L(\zeta)$ for all $\zeta$, which identifies the limit.
Now it is standard to pass to the limit $m \to \infty$ in \cref{Eq:GalerkinM} to deduce that the limit $(\ul,\hpsil)$ (with the previously suppressed superscript $.^{\ell}$ reinstated in our notation) satisfies the $\l$-truncated system \cref{Def:Truncation_Velocity}, \cref{Def:Truncation_PDF}.

It remains to prove the energy inequality stated in \cref{Lem:ExistenceLapprox}. We shall do so by passing to the limit in all terms on the left-hand side of the $m$-uniform energy estimate \cref{Eq:Entropy}, using the weak lower semicontinuity of the norms and Fatou's lemma for the nonnegative function $G$. Furthermore, for the term involving the initial data on the right-hand side of the inequality, we notice that
$$
\|M^{m} G(T_{\l}(\hpsi_{0}^{m}))\|_{L^{1}(\O)} \to\|M G(T_{\l}(\hpsi_{0}))\|_{L^{1}(\O)} \leq C .
$$
Thus, we obtain the energy inequality \cref{Eq:EnergyLexistenc}, and we have proved \cref{Lem:ExistenceLapprox}.
\end{proof}

\subsection{Step 4: Uniform a priori estimates with respect to the truncation parameter $\ell$ and the final limit process $\ell \to \infty$}
In this last step, we shall derive $\l$-independent bounds, which will enable us to pass to the limit $\l \to \infty$ to prove \cref{Thm:Main}, that is, the existence of a weak solution to the nonlocal Navier--Stokes--Fokker--Planck system \cref{Def:System}. To this end, we reinstate the omitted index $\ell$.

\begin{proof}[Proof of \cref{Thm:Main}]
    We begin by recalling the uniform estimate \cref{Eq:EnergyLexistenc} stated in  \cref{Lem:ExistenceLapprox}, according to which
\begin{equation}\begin{aligned}
    &\|\ul\|_{L^\infty(0,T;L^2(\Omega))}^2 +  \|\nabla \ul\|_{L^2(\Omega_T)}^2 + \|\hpsil\ln\hpsil\|_{L^1_{M}(\O_T)}+\Big\|\sqrt{\hpsil} \Big\|_{L^2(0,T;H^1_{M}(\O))}^2 \\ &\quad + \|\rhol\|_{L^\infty(\Omega_T)} 
    + \|\mathbb{S}_\l(\hpsil)\|_{L^2(\Omega_T)}^2 \leq C;
\end{aligned}\label{Eq:Lestimate} \end{equation}
recall that $\rhol:=(M,\hpsil)_D$. As before, we obtain by function space interpolation that 
$\ul$ is bounded in $L^{2(d+2)/d}(\Omega_T)$ uniformly in $\ell$.
We consider an arbitrary function $w \in L^{4/(4-d)}(0,T;W^{1,2}_{0,\div  }(\Omega))$; such a $w$ is a valid test function in the $\l$-truncated Navier--Stokes system \cref{Def:Truncation_Velocity}. Thus, testing with $w$ we find that
\begin{align*}\begin{aligned}|\langle\pt \ul,&w \rangle|
=\big|(\Gamma_\l(|\ul|^2) \ul \otimes \ul- \nabla \ul- S_\ell(\hpsil),\nabla  w)_{\Omega_T} + (f,w)_{\Omega_T}\big|
\\&\leq C\|\Gamma_\l(|\ul|^2)\|_\infty \|\ul(s)\|_{L^\infty(0,T;L^2(\Omega))}^{2-d/2} \|\nabla\ul(s)\|_{L^2(\Omega_T)}^{d/2} \|\nabla w\|_{L^{4/(4-d)}(0,T;L^2(\Omega))}
\\[-0.1cm] &\hspace{-0.4cm} + \Big(\|\nabla \ul\|_{L^2(\Omega_T)} + \|\mathbb{S}_\l(\hpsil)\|_{L^2(\Omega_T)}+\|f\|_{L^2(\I;[W^{1,2}_{0,\div  }(\Omega)]^*)}\Big) \|w\|_{L^2(\I;W^{1,2}_{0,\div  }(\Omega))} \\
&\leq C\|w\|_{L^{4/(4-d)}(0,T;W^{1,2}_{0,\div  }(\Omega))},
\end{aligned}
\end{align*}
where we have used the $\l$-uniform bounds together with the uniform bound $\|\Gamma_\l(|\ul|^2)\|_\infty \leq 1$ in the last step.
Thus, taking the supremum over all functions in $L^{4/(4-d)}(0,T;W^{1,2}_{0,\div  }(\Omega))$ gives
$$\|\pt \ul \|_{L^{4/d}(0,T;[W^{1,2}_{0,\div  }(\Omega)]^*)} \leq C.$$
Using the $\l$-uniform bounds and applying the Banach--Alaoglu theorem and the Aubin--Lions lemma \cref{Eq:aubinclassic}, we obtain the existence of a subsequence $(\ul,\hpsil)$ (not relabelled) and limit functions $(u,\hpsi)$ that satisfy the following convergences:
    $$\begin{aligned}
        \ul &\overset{*}{\rightharpoonup} u &&\text{weakly in } L^\infty(\I;L^2_{0,\div  }(\Omega))\! \cap \!L^2(\I;W^{1,2}_{0,\div  }(\Omega))\!\cap\! L^{2(d+2)/d}(\Omega_T)^d, \\
        \pt \ul &\rightharpoonup \pt u &&\text{weakly in } L^{4/d}(\I;[W^{1,2}_{0,\div  }(\Omega)]^*), \\
        \ul &\to u &&\text{strongly in } L^2(\I;L^r(\Omega)^d) \cap C([0,T];[W^{1,2}_{0,\div  }(\Omega)]^*), \\
        \mathbb{S}_\l(\hpsil) &\rightharpoonup \mathbb{S}(\hpsi) &&\text{weakly in } L^2(\Omega_T)^{d\times d},
    \end{aligned}$$
    where $r<2^*$ is less than the Sobolev conjugate $2^*$ of $2$, ensuring the compactness of the first of the two embeddings involved in the Gelfand triple $W^{1,2}_{0,\div  }(\Omega) \dhookrightarrow L^r(\Omega) \con [W^{1,2}_{0,\div  }(\Omega)]^*$ in the application of the Aubin--Lions lemma.
With these convergence results, we can pass to the limit $\l \to \infty$ in the $\l$-truncated Navier--Stokes equation \cref{Def:Truncation_Velocity} to deduce that $(u,\psi)$ solves the variational form of the Navier--Stokes equation \cref{Def:SysVar}$_1$ as stated in \cref{Def:Weak}. It remains to prove that $(u,\psi)$ also satisfies the nonlocal Fokker--Planck equation \cref{Def:SysVar}$_2$.

To deduce an essential strong convergence result for $\hpsil$, we mimic the procedure from the proof of \cref{Lem:StrongPsi} where we obtained the strong convergence of $\hpsi^{m,\ell}$ by means of the nonlocal compactness result \cref{thm:nonlocal_aubin_measures}. In fact, we begin by considering the nonlocal Fokker--Planck equation \cref{Def:SysVar} for $M\pt(k*[\hpsil-\hpsil_0])$, integrate in time from $0$ to $T$ to infer for every test function $\zeta \in C([\I];W^{1,\infty}(\O))$ that
$$\begin{aligned}
&\langle M\pt(k*[\hpsil-\hpsil_0]),\zeta \rangle\\ &=(M\ul\hpsil,\nabla \zeta)_{\O_T}\!  - (M\nabla\hpsil,\nabla \zeta)_{\O_T}\!-(M\nablaq \hpsil,\nablaq \zeta)_{\O_T}\! +(M\Lambda_\l(\hpsil)(\nabla \ul) q,\nablaq \zeta)_{\O_T}\\
&\leq \|M\|_\infty \|\ul\|_{L^\infty(0,T;L^2(\Omega))}  \|\hpsil\|_{L^\infty(\Omega_T;L^1(D))} \|\nabla\zeta\|_{L^1(0,T;L^\infty(\O))}   \\ &\quad+ \|M\|_\infty\|\nabla_{x,q} \hpsil\|_{L^1(\O_T)} \|\nabla_{x,q} \zeta\|_{C([\I];L^\infty(\O))} \\ &\quad  + \|M\|_\infty \|\Lambda_\l(\hpsil)\|_{L^\infty(\Omega_T;L^1(D))} \|\nabla \ul\|_{L^2(\Omega_T)} \|q\|_\infty  \|\nablaq \zeta\|_{C([\I];L^\infty(\O))}.
\end{aligned}$$
Using the above $\l$-uniform bounds and the definition of $\Lambda_\l$, we can take supremum over all test function and obtain that
$$
\| M\pt(k*[\hpsil-\hpsil_0])\|_{[C([\I];W^{1,\infty}(\O))]^*} \leq C.
$$
Hence, because $M$ is independent of $t$ and $M(\hpsil-\hpsil_0) =\psi^\ell - \psi^{\ell}_0$, it follows that  
\[  \|\pt(k*[\psi^\l-\psi^{\l}_0])\|_{[C([\I];W^{1,\infty}(\O))]^*} \leq C.\]
Because $[C([\I];W^{1,\infty}(\O))]^* = \mathcal{M}(0,T;[W^{1,\infty}(\Omega)]^*)$,  we can apply the compactness result \cref{thm:nonlocal_aubin_measures} to infer the strong convergence
$$\psi^\l \to \psi  \quad\text { strongly in } L^1(\O_T).$$
Thus, in summary, by mimicking the arguments from the previous section, we deduce the following convergence results by interpolating similarly as in the proof of \cref{Lem:StrongPsi}:
$$
\begin{aligned}
\hpsil &\to \hpsi  &&\text { strongly in } L^{s}(\Omega_T ; L^{1}_M(\D)), \\
 \nablaxq\hpsil &\rightharpoonup \nablaxq\hpsi  &&\text { weakly in } L^{1}(\Omega_T;L^1_M(D)^{d(K+1)}), \\
	 \hpsil \ul &\rightharpoonup \hpsi u && \text { strongly in } L^r(\Omega_T ; L^{1}_M(\D)^{d(K+1)}), \\
	\Lambda_{\l}(\hpsil) \nabla \ul &\rightharpoonup \hpsil \nabla u && \text { weakly in } L^{2}(\Omega_T;L^1_M(\D)^{d \times d}), \\
    \pt (k*[\hpsil-\hpsil_0]) &\overset{*}{\rightharpoonup} \pt(k*[\hpsi-\hpsi_0]) &&\text{ weakly-$*$ in } \mathcal{M}(0,T;[W^{1,\infty}(\O)]^*), 
\end{aligned}
$$
for any $s \in [1,\infty)$ and any $r \in [1,2^*)$, where $2^*:=2d/(d-2)$.
We can now pass to the limit $\l \to \infty$ in \cref{Def:Truncation_Velocity}, \cref{Def:Truncation_PDF} to conclude that $(u,\psi)$ solves the variational form of the Navier--Stokes--Fokker--Planck system as stated in \cref{Def:Weak}. The verification of the 
 attainment of the initial condition follows the same line of argument as before.
\end{proof}

\subsection{Step 5: Uniqueness}

\begin{proof}[Proof of \cref{Thm:WeakStrong}] Let us consider two weak solutions $(u_1,\hpsi_1)$ and $(u_2,\hpsi_2)$ in the sense of \cref{Def:Weak} that satisfy the regularity assumptions stated in \cref{Thm:WeakStrong}. We consider their difference $(u,\hpsi):=(u_1-u_2,\hpsi_1-\hpsi_2)$ and note that this pair of functions satisfies the following variational problem:
\begin{equation}\label{Def:SysVarUni}\begin{aligned}
        &\langle \pt u,w \rangle -(u_1 \otimes u,\nabla w)_{\Omega}-(u \otimes u_2,\nabla w)_{\Omega} +  (\nabla u,\nabla w)_{\Omega}  +(S\hpsi,\nabla w)_{\Omega} =0, \\[.1cm]
         &\langle M\pt(k*\hpsi),\zeta \rangle + (M \nabla \hpsi,\nabla\zeta)_{\O}+ (M\nablaq \hpsi,\nablaq \zeta)_{\O}- k*(Mu_1\hpsi,\nabla\zeta)_{\O}& \\&\,\,\,  - k*(Mu\hpsi_2,\nabla\zeta)_{\O}   -  k*(M\hpsi_1 (\nabla u) q,\nablaq \zeta)_{\O}-  k*(M\hpsi (\nabla u_2) q,\nablaq \zeta)_{\O} =0 , 
    \end{aligned}\end{equation}
    for the same choice of test functions $w$ and $\zeta$ as in Definition \ref{Def:Weak}.
    Setting $w=u$  in \cref{Def:SysVarUni}$_1$ yields that \vspace{-.4cm}
\begin{equation}\label{Eq:UniqueU}\begin{aligned}&\frac12\ddt \|u\|_{L^2(\Omega)}^2  +  \|\nabla u\|^2_{L^2(\Omega)}  \\ &\quad =-((u \cdot \nabla) u_2,u)_{\Omega}-(\mathbb{S}(\hpsi),\nabla u)_{\Omega} \\ &\quad =((u \cdot \nabla) u,u_2)_{\Omega}-(\mathbb{S}(\hpsi),\nabla u)_{\Omega} \\
    &\quad \leq \frac12 \|\nabla u\|^2_{L^2(\Omega)} + C\|u\|_{L^2(\Omega)}^2 \|u_2\|_{L^4(\Omega)}^8+Cq_\infty^2 \|\nablaq \hpsi\|_{L^2_M(\O)}^2, 
    \end{aligned}\end{equation}
    where we made use of the specific form of the extra-stress tensor $S$, see \cref{Def:TensorS}.
    Furthermore, choosing $\zeta=\hpsi$ in \cref{Def:SysVarUni}$_2$ gives with Alikhanov's, H\"{o}lder's, Young's inequalities\vspace{-.2cm}
\begin{equation}\label{Eq:UniquePsi}\begin{aligned}
        &\frac12 \ddt \big(k*\|\hpsi\|_{L^2_M(\O)}^2\big) + \|\nablaxq \hpsi\|_{L^2_M(\O)}^2  \\ &\quad\leq (Mu\hpsi_2,\nabla\hpsi)_{\O}   + (M\hpsi_1 (\nabla u) q,\nablaq \hpsi)_{\O}+  (M\hpsi (\nabla u_2) q,\nablaq \hpsi)_{\O}
        \\ &\quad \leq C\|\hpsi_2\|_\infty^2\|u\|_{L^2(\Omega)}^2+\tfrac{\tilde k(T)}{2} \|\nabla\hpsi\|_{L^2_M(\O)}^2   + q_\infty\|\hpsi_1\|_\infty^2 \|\nabla u\|_{L^2(\Omega)}^2 \\[-.1cm] & \qquad + q_\infty \|\nablaq\hpsi\|_{L^2_M(\O)}^2+ q_\infty \|\hpsi\|_{L^6(\O)} \|\nabla u_2\|_{L^3(\Omega)} \|\nablaq \hpsi\|_{L^2_M(\O)} .
    \end{aligned}\end{equation}
    At this point, we integrate \cref{Eq:UniqueU} over $(0,t)$ and convolve \cref{Eq:UniquePsi} with the resolvent kernel $\tilde k$, use the inverse convolution property \cref{FracProp} and add the inequalities to obtain \vspace{-.2cm}
$$\begin{aligned}
    & \frac12 \|u(t)\|_{L^2(\Omega)}^2  +  \Big(\frac12-q_\infty\|\hpsi_1\|_{L^\infty(0,T;L^\infty_M(\O))}^2\Big) \|\nabla u\|^2_{L^2(0,t;L^2(\Omega))}+\frac12 \|\hpsi(t)\|_{L^2_M(\O)}^2 \\ &\quad +\Big(\frac{\tilde k(T)}{2} -q_\infty-Cq_\infty^2-Cq_\infty\|\nabla u_2\|_{L^\infty(0,T;L^3(\Omega))}\Big) \|\nablaxq \hpsi\|_{L^2(0,t;L^2_M(\O))}^2 \\
    &\leq C\int_0^t \big(\|u_2(s)\|_{L^4(\Omega)}^8+\|\hpsi_2\|_\infty^2\big) \|u(s)\|_{L^2(\Omega)}^2 \ds  
\end{aligned}$$ 
Since $q_\infty$ is assumed to be sufficiently small, Gronwall's inequality with an integrable prefactor in the integral on the right-hand side yields the desired result.
\end{proof}

%% file: 4_outlook.tex
\section{Outlook}\label{Sec:Outlook}
We have established the existence of large-data global-in-time weak solutions to a nonlocal Navier--Stokes--Fokker--Planck system describing dilute polymeric fluids with anomalous diffusion. The analysis successfully addresses three key challenges inherent to the coupled system: (i) the nonlocal temporal coupling between the macroscopic momentum balance and mezoscopic polymer configuration evolution equation, resolved through careful energy-entropy estimates; (ii) preservation of the physical structure via a maximum-principle argument, ensuring nonnegativity of the probability density function, in conjunction with an energy inequality; and (iii) development of a compactness argument using a novel nonlocal version of the Aubin--Lions lemma. The proof of uniqueness of weak solutions that possess  sufficient regularity completes the well-posedness theory, while our constructive approximation scheme provides a potential pathway for the construction of convergent numerical approximation methods. These results form a rigorous framework for the mathematical analysis of models of subdiffusive polymeric fluids that exhibit memory effects.

Several promising research directions emerge from this work. First, the development of structure-preserving numerical methods for nonlocal operators remains crucial, particularly discretisations that maintain entropy dissipation properties and preservation of the physically relevant property of finite extension of polymer molecules in FENE-type models. Probabilistic and Monte Carlo-type approaches for fractional PDEs have recently been explored in \cite{Kolokoltsov2021}, suggesting promising numerical frameworks for systems with memory. Second, the influence of the choice of the temporal fractional derivative (Caputo vs. Riemann--Liouville) on the behaviour of solutions to these models warrants detailed analytical and numerical investigation. Previous publications on time-fractional dilute polymer systems considered Hookean-type models, see \cite{beddrich2024numerical,beddrich2025numerical}, while we are mainly interested in the class of, physically more realistic, FENE-type models.  

%% file: literature.bib
@article{angstmann2015generalized,
    author = {Angstmann, C. N. and Donnelly, I. C. and Henry, B. I. and Langlands, T. A. M. and Straka, P.},
    title = {Generalized continuous time random walks, master equations, and fractional {Fokker--Planck} equations},
    journal = {SIAM J. Appl. Math.},
    volume = {75},
    number = {4},
    pages = {1445--1468},
    year = {2015},
    doi = {10.1137/14099048X}
}

@book{ambrosio2000functions,
  title={Functions of Bounded Variation and Free Discontinuity Problems},
  author={Ambrosio, Luigi and Fusco, Nicola and Pallara, Diego},
  year={2000},
  publisher={Oxford University Press}
}

@book{brezis2011functional,
  title={Functional Analysis, Sobolev Spaces and Partial Differential Equations},
  author={Br{\'e}zis, Haim},
  year={2011},
  publisher={Springer}
}

@book{atanackovic2014fractional,
    author = {Atanackovi\'c, T. M. and Pilipovi\'c, S. and Stankovic, B. and Zorica, D.},
    title = {Fractional Calculus with Applications in Mechanics: Wave Propagation, Impact and Variational Principles},
    publisher = {John Wiley \& Sons},
    year = {2014},
    doi = {10.1002/9781118784508}
}

@article{barrett2009numerical,
    author = {Barrett, J. W. and S{\"u}li, E.},
    title = {Numerical approximation of corotational dumbbell models for dilute polymers},
    journal = {IMA J. Numer. Anal.},
    volume = {29},
    number = {4},
    pages = {937--959},
    year = {2009},
    doi = {10.1093/imanum/drn049}
}

@article{barrett2011existence,
    author = {Barrett, J. W. and S{\"u}li, E.},
    title = {Existence and equilibration of global weak solutions to kinetic models for dilute polymers {I}: {Finitely} extensible nonlinear bead-spring chains},
    journal = {Math. Models Methods Appl. Sci.},
    volume = {21},
    number = {06},
    pages = {1211--1289},
    year = {2011},
    doi = {10.1142/S0218202511005225}
}

@article{barrett2011finite,
    author = {Barrett, J. W. and S{\"u}li, E.},
    title = {Finite element approximation of kinetic dilute polymer models with microscopic cut-off},
    journal = {ESAIM: Math. Model. Numer. Anal.},
    volume = {45},
    number = {1},
    pages = {39--89},
    year = {2011},
    doi = {10.1051/m2an/2010030}
}

@article{barrett2012finite,
    author = {Barrett, J. W. and S{\"u}li, E.},
    title = {Finite element approximation of finitely extensible nonlinear elastic dumbbell models for dilute polymers},
    journal = {ESAIM: Math. Model. Numer. Anal.},
    volume = {46},
    number = {4},
    pages = {949--978},
    year = {2012},
    doi = {10.1051/m2an/2011062}
}

@book{bird1987dynamics,
    author = {Bird, R. B. and Curtiss, C. F. and Armstrong, R. C. and Hassager, O.},
    title = {Dynamics of Polymeric Liquids, Volume 2: Kinetic Theory},
    publisher = {Wiley},
    year = {1987}
}

@article{bulicek2013existence,
    author = {Bul\'{\i}\v{c}ek, M. and M{\'a}lek, J. and S{\"u}li, E.},
    title = {Existence of global weak solutions to implicitly constituted kinetic models of incompressible homogeneous dilute polymers},
    journal = {Commun. Partial Differ. Equ.},
    volume = {38},
    number = {5},
    pages = {882--924},
    year = {2013},
    doi = {10.1080/03605302.2012.747705}
}

@book{dellacherie1978probabilities,
    author = {Dellacherie, Claude and Meyer, Paul-Andr{\'e}},
    title = {Probabilities and Potentials},
    publisher = {North-Holland},
    year = {1978},
    isbn = {072040701X}
}

@book{fallahgoul2016fractional,
    author = {Fallahgoul, Hasan and Focardi, Sergio and Fabozzi, Frank},
    title = {Fractional Calculus and Fractional Processes with Applications to Financial Economics: Theory and Application},
    publisher = {Academic Press},
    year = {2016},
    doi = {10.1016/C2015-0-01989-4},
    isbn = {9780128042489}
}

@article{fritz2021subdiffusive,
    author = {Fritz, Marvin and Kuttler, Christina and Rajendran, Mabel L and Wohlmuth, Barbara and Scarabosio, Laura},
    title = {On a subdiffusive tumour growth model with fractional time derivative},
    journal = {IMA J. Appl. Math.},
    volume = {86},
    number = {4},
    pages = {688--729},
    year = {2021},
    doi = {10.1093/imamat/hxab009}
}

@article{fritz2022time,
    author = {Fritz, Marvin and Rajendran, Mabel L and Wohlmuth, Barbara},
    title = {Time-fractional {C}ahn--{H}illiard equation: Well-posedness, regularity, degeneracy, and numerical solutions},
    journal = {Comput. Math. Appl.},
    volume = {108},
    pages = {66--87},
    year = {2022},
    doi = {10.1016/j.camwa.2022.01.002}
}

@article{fritz2024analysis,
    author = {Fritz, Marvin and S{\"u}li, Endre and Wohlmuth, Barbara},
    title = {Analysis of a dilute polymer model with a time-fractional derivative},
    journal = {SIAM J. Math. Anal.},
    volume = {56},
    pages = {2063--2089},
    year = {2024},
    doi = {10.1137/23M1590767}
}

@article{heinsalu2007use,
    author = {Heinsalu, Els and Patriarca, Marco and Goychuk, Igor and H{\"a}nggi, Peter},
    title = {Use and abuse of a fractional {F}okker--{P}lanck dynamics for time-dependent driving},
    journal = {Phys. Rev. Lett.},
    volume = {99},
    number = {12},
    pages = {120602},
    year = {2007},
    doi = {10.1103/PhysRevLett.99.120602}
}

@article{magdziarz2008equivalence,
    author = {Magdziarz, Marcin and Weron, Aleksander and Klafter, Joseph},
    title = {Equivalence of the fractional {F}okker--{P}lanck and subordinated {L}angevin equations: the case of a time-dependent force},
    journal = {Phys. Rev. Lett.},
    volume = {101},
    number = {21},
    pages = {210601},
    year = {2008},
    doi = {10.1103/PhysRevLett.101.210601}
}

@book{mainardi2022fractional,
    author = {Mainardi, Francesco},
    title = {Fractional Calculus and Waves in Linear Viscoelasticity: An Introduction to Mathematical Models},
    publisher = {World Scientific},
    year = {2022},
    doi = {10.1142/p614}
}

@article{meliani2023unified,
    author = {Meliani, Mostafa},
    title = {A unified analysis framework for generalized fractional {M}oore--{G}ibson--{T}hompson equations: Well-posedness and singular limits},
    journal = {Fract. Calc. Appl. Anal.},
    volume = {26},
    number = {6},
    pages = {2540--2579},
    year = {2023},
    doi = {10.1515/fca-2023-0127}
}

@book{ottinger2012stochastic,
    author = {{\"O}ttinger, Hans Christian},
    title = {Stochastic Processes in Polymeric Fluids: Tools and Examples for Developing Simulation Algorithms},
    publisher = {Springer},
    year = {2012},
    doi = {10.1007/978-3-642-58290-5}
}

@book{pilipovic2014fractional,
    author = {Pilipovi\'c, Stevan and Atanackovi\'c, Teodor M and Stankovi\'c, Bogoljub and Zorica, Dusan},
    title = {Fractional Calculus with Applications in Mechanics: Vibrations and Diffusion Processes},
    publisher = {John Wiley \& Sons},
    year = {2014},
    doi = {10.1002/9781118577530}
}

@book{roubicek2013nonlinear,
    author = {Roubi\v{c}ek, Tomas},
    title = {Nonlinear Partial Differential Equations with Applications},
    publisher = {Springer},
    year = {2013},
    doi = {10.1007/978-3-0348-0513-1}
}

@article{vergara2008lyapunov,
    author = {Vergara, Vicente and Zacher, Rico},
    title = {Lyapunov functions and convergence to steady state for differential equations of fractional order},
    journal = {Math. Z.},
    volume = {259},
    pages = {287--309},
    year = {2008},
    doi = {10.1007/s00209-007-0225-1}
}

@book{yang2020general,
    author = {Yang, Xiao-Jun and Gao, Feng and Yang, Jun},
    title = {General Fractional Derivatives with Applications in Viscoelasticity},
    publisher = {Academic Press},
    year = {2020}
}

@article{zacher2008boundedness,
    author = {Zacher, Rico},
    title = {Boundedness of weak solutions to evolutionary partial integro-differential equations with discontinuous coefficients},
    journal = {J. Math. Anal. Appl.},
    volume = {348},
    number = {1},
    pages = {137--149},
    year = {2008},
    doi = {10.1016/j.jmaa.2008.06.054}
}

@article{beddrich2024numerical,
  title={Numerical simulation of the time-fractional {Fokker--Planck} equation and applications to polymeric fluids},
  author={Beddrich, Jonas and S{\"u}li, Endre and Wohlmuth, Barbara},
  journal={J. Comp. Phys.},
  volume={497},
  pages={112598},
  year={2024},
  publisher={Elsevier},
doi ={10.1016/j.jcp.2023.112598}
}

@article{beddrich2025numerical,
  title={Numerical simulation of dilute polymeric fluids with memory effects in the turbulent flow regime},
  author={Beddrich, Jonas and Lunowa, Stephan B and Wohlmuth, Barbara},
  journal={J. Comp. Phys.},
  pages={113955},
  year={2025},
  publisher={Elsevier},
doi={10.1016/j.jcp.2025.113955}
}

@article{escauriaza2003l_3,
  title={${L}_{3,\infty}$-solutions of the {Navier--Stokes} equations and backward uniqueness},
  author={Escauriaza, Luis and Seregin, Grigorii Aleksandrovich and {\v{S}}ver{\'a}k, Vladim{\'\i}r},
  journal={Russ. Math. Surv.},
  volume={58},
  number={2},
  pages={211--250},
  year={2003},
doi ={10.1070/RM2003v058n02ABEH000609}
}

@article{li2007mathematical,
  title={Mathematical analysis of multi-scale models of complex fluids},
  author={Li, Tiejun and Zhang, Pingwen},
  year={2007},
journal = {Commun. Math. Sci.},
volume={5},
pages={1--51}
}

@article{fritz2022equivalence,
  title={Equivalence between a time-fractional and an integer-order gradient flow: {T}he memory effect reflected in the energy},
  author={Fritz, Marvin and Khristenko, Ustim and Wohlmuth, Barbara},
  journal={Adv. Nonlinear Anal.},
  volume={12},
  number={1},
  pages={20220262},
  year={2022},
  doi ={10.1515/anona-2022-0262},
  publisher={De Gruyter}
}

@article{kufner1984define,
  title={How to define reasonably weighted {S}obolev spaces},
  author={Kufner, Alois and Opic, Bohum{\'\i}r},
  journal={Comment. Math. Univ. Carolin.},
  volume={25},
  number={3},
  pages={537--554},
  year={1984},
  publisher={Charles University in Prague, Faculty of Mathematics and Physics}
}

@article{simon1986compact,
  title={Compact sets in the space {$L^p(0,T;B)$}},
  author={Simon, Jacques},
  journal={Ann. Mat. Pura Appl.},
  volume={146},
  number={1},
  pages={65--96},
  year={1986},
  publisher={Springer},
doi={10.1007/BF01762360}
}

@book{kirkwood1967Macromolecules,
    author = {Kirkwood, J. G.},
    title = {Macromolecules},
    publisher = {Gordon and Breach},
    year = 1967
}

@article{kalita2013convergence,
  title={Convergence of {R}othe scheme for hemivariational inequalities of parabolic type},
  author={Kalita, Piotr},
  journal={Int. J. Numer. Anal. Model.},
  volume={10},
  number={2},
  year={2013}
}

@article{cascales1991simulation,
  title={Simulation of polymer chains in elongational flow. Steady-state properties and chain fracture},
  author={Cascales, JJ Lopez and de la Torre, J Garcia},
  journal={J. Chem. Phys.},
  volume={95},
  number={12},
  pages={9384--9392},
  year={1991},
  publisher={American Institute of Physics}
}

@book{jin2021fractional,
  title={Fractional Differential Equations},
  author={Jin, Bangti},
  year={2021},
  publisher={Springer}
}

@book{kubica2020time,
  title={Time-Fractional Differential Equations: A Theoretical Introduction},
  author={Kubica, Adam and Ryszewska, Katarzyna and Yamamoto, Masahiro},
  year={2020},
  publisher={Springer}
}

@article{gol2009weighted,
  title={Weighted {S}obolev spaces and embedding theorems},
  author={Gol’dshtein, Vladimir and Ukhlov, Alexander},
  journal={Trans. Amer. Math. Soc.},
  volume={361},
  number={7},
  pages={3829--3850},
  year={2009}
}

@article{li2018some,
  title={Some compactness criteria for weak solutions of time fractional {PDE}s},
  author={Li, Lei and Liu, Jian-Guo},
  journal={SIAM J. Math. Anal.},
  volume={50},
  number={4},
  pages={3963--3995},
  year={2018},
doi={10.1137/17M1145549},
  publisher={SIAM}
}

@article{fritz2024well,
  title={Well-posedness and simulation of weak solutions to the time-fractional {Fokker--Planck} equation with general forcing},
  author={Fritz, Marvin},
  journal={Discrete Contin. Dyn. Syst. Ser. B},
  volume={29},
  number={10},
  pages={4097--4119},
  year={2024},
  doi = {10.3934/dcdsb.2024036}
}

@article{Weron2008ModelingEquation,
    title = {{Modeling of subdiffusion in space-time-dependent force fields beyond the fractional Fokker-Planck equation}},
    year = {2008},
    journal = {Phys. Rev. E},
    author = {Weron, Aleksander and Magdziarz, Marcin and Weron, Karina},
    number = {3},
pages=036704,
    month = {3},
    volume = {77},
    doi = {10.1103/PhysRevE.77.036704},
    issn = {15393755}
}

@book{gripenberg1990volterra,
  title={Volterra Integral and Functional Equations},
  author={Gripenberg, Gustaf and Londen, Stig-Olof and Staffans, Olof},
  year={1990},
  publisher={Cambridge University Press},
doi ={10.1017/CBO9780511662805}
}

@book{diethelm2010analysis,
	title        = {{The Analysis of Fractional Differential Equations: An Application-Oriented Exposition using Differential Operators of Caputo type}},
	author       = {Diethelm, Kai},
	year         = 2010,
	publisher    = {Springer},
doi ={10.1007/978-3-642-14574-2}
}

@article{Kemppainen2016a,
    title = {{Decay estimates for time-fractional and other non-local in time subdiffusion equations in $\mathbb{R}^d$}},
    year = {2016},
    journal = {Math. Ann.},
    author = {Kemppainen, Jukka and Siljander, Juhana and Vergara, Vicente and Zacher, Rico},
    number = {3-4},
    month = {12},
    pages = {941--979},
    volume = {366},
    publisher = {Springer New York LLC},
    doi = {10.1007/s00208-015-1356-z},
    issn = {00255831},
    arxivId = {1403.1737}
}

@article{HahnUmarov2011,
  author    = {Michael Hahn and Sabir Umarov},
  title     = {Fractional {Fokker--Planck--Kolmogorov} type Equations and their Associated Stochastic Differential Equations},
  journal   = {Frac. Calc. Appl. Anal.},
  volume    = {14},
  number    = {1},
  pages     = {56--79},
  year      = {2011},
doi={10.2478/s13540-011-0005-9}
}

@article{Sandev2015,
  author    = {T. Sandev and A. Chechkin and H. Kantz and R. Metzler},
  title     = {Diffusion and {Fokker--Planck--Smoluchowski} Equations with Generalized Memory Kernel},
  journal   = {Frac. Calc. Appl. Anal.},
  volume    = {18},
  number    = {4},
  pages     = {1006--1038},
  year      = {2015},
doi={10.1515/fca-2015-0059},
}

@article{Kolokoltsov2021,
  author    = {Vassili Kolokoltsov and Fang Lin and Aleksandar Mijatović},
  title     = {Monte {C}arlo Estimation of the Solution of Fractional {PDEs}},
  journal   = {Frac. Calc. Appl. Anal.},
doi={10.1515/fca-2021-0012},
  volume    = {24},
  number    = {1},
  pages     = {278--306},
  year      = {2021}
}

@article{Hanyga2020,
  author    = {A. Hanyga},
  title     = {A Generalized Fractional Derivative Cannot Have a Regular Kernel},
  journal   = {Frac. Calc. Appl. Anal.},
doi={10.1515/fca-2020-0008},
  volume    = {23},
  number    = {1},
  pages     = {211--223},
  year      = {2020}
}

@article{AwadMetzler2020,
  author    = {E. Awad and R. Metzler},
  title     = {Crossover Dynamics from Superdiffusion to Subdiffusion: Models and Solutions},
  journal   = {Frac. Calc. Appl. Anal.},
  volume    = {23},
  number    = {1},
  pages     = {55--102},
  year      = {2020},
doi={10.1515/fca-2020-0003}
}

@book{Boyer2013,
    title = {{Mathematical Tools for the Study of the Incompressible Navier--Stokes Equations and Related Models}},
    year = {2013},
    author = {Boyer, Franck and Fabrie, Pierre},
    publisher={Springer},
    doi = {10.1007/978-1-4614-5975-0},
}

@article{Wittbold2020a,
    title = {{Bounded weak solutions of time-fractional porous medium type and more general nonlinear and degenerate evolutionary integro-differential equations}},
    year = {2021},
    journal = {J. Math. Anal. Appl.},
    author = {Wittbold, Petra and Wolejko, Patryk and Zacher, Rico},
    number = {1},
    month = {8},
    volume = {499},
pages=125007,
    publisher = {Academic Press Inc.},
    doi = {10.1016/j.jmaa.2021.125007},
    issn = {10960813},
    arxivId = {2008.10919}
}

@book{Fonseca2006,
    title = {{Modern Methods in the Calculus of Variations: $L_p$ Spaces}},
    year = {2006},
    publisher={Springer},
    author = {Fonseca, Irene and Leoni, Giovanni},
    doi = {10.1007/978-0-387-69006-3}
}

@book {MR153974,
    AUTHOR = {Lions, J.-L.},
     TITLE = {\'{E}quations Diff\'{e}rentielles Op\'{e}rationnelles et Probl\`emes aux
              Limites},
 PUBLISHER = {Springer},
      YEAR = {1961},
     PAGES = {ix+292},
   MRCLASS = {35.00 (34.95)},
  MRNUMBER = {153974},
MRREVIEWER = {S. Zaidman},
}

@book {MR259693,
    AUTHOR = {Lions, J.-L.},
     TITLE = {Quelques M\'{e}thodes de R\'{e}solution des Probl\`emes aux Limites Non
              Lin\'{e}aires},
 PUBLISHER = {Dunod},
      YEAR = {1969},
     PAGES = {xx+554},
   MRCLASS = {47.80 (35.00)},
  MRNUMBER = {259693},
MRREVIEWER = {L. Cesari},
}

@article {Olsen-Holden,
    AUTHOR = {Hanche-Olsen, Harald and Holden, Helge},
     TITLE = {The {A}ubin--{L}ions--{D}ubinski\u{\i} theorems on compactness in
              {B}ochner spaces},
   JOURNAL = {Pure Appl. Funct. Anal.},
  FJOURNAL = {Pure and Applied Functional Analysis},
    VOLUME = {9},
      YEAR = {2024},
    NUMBER = {5},
     PAGES = {1133--1144},
      ISSN = {2189-3756},
   MRCLASS = {46B50 (46E30 46E35 46N20)},
  MRNUMBER = {4852582},
MRREVIEWER = {Aigerim Kalybay},
}

@book {MR453964,
    AUTHOR = {Diestel, J. and Uhl, Jr., J. J.},
     TITLE = {Vector Measures},
 PUBLISHER = {American Mathematical Society},
      YEAR = {1977},
     PAGES = {xiii+322},
   MRCLASS = {28A45 (46B05 46G10)},
  MRNUMBER = {453964},
MRREVIEWER = {Robert E. Huff},
}
